\documentclass[11pt]{amsart}

\usepackage[numbers]{natbib}
\usepackage{amsmath,amssymb,amsfonts,amsthm}
\usepackage{geometry}
\usepackage{graphicx}
\usepackage{booktabs}
\usepackage{bm}
\usepackage{longtable}
\usepackage{tikz}
\usetikzlibrary{positioning,fit,backgrounds}
\usepackage{xcolor}
\usepackage{hyperref}


\newcommand{\R}{\mathbb{R}}
\newcommand{\E}{\mathbb{E}}
\newcommand{\Pprob}{\mathbb{P}}
\newcommand{\Filt}{\mathcal{F}}
\newcommand{\Gfilt}{\mathcal{G}}
\newcommand{\Tr}{\operatorname{Tr}}
\newcommand{\diff}{\mathrm{d}}

\newcommand{\Sym}{\mathbb{S}}
\newcommand{\Acal}{\mathcal{A}}
\newcommand{\Ucal}{\mathcal{U}}
\newcommand{\Hcal}{\mathcal{H}}
\newcommand{\Ical}{\mathcal{I}}
\newcommand{\Dcal}{\mathcal{D}}
\newcommand{\Lcal}{\mathcal{L}}
\newcommand{\Vcal}{\mathcal{V}}

\newcommand{\Pmax}{P^{\max}}

\newcommand{\Nu}{\mathcal{N}}

\theoremstyle{plain}
\newtheorem{theorem}{Theorem}
\newtheorem{lemma}{Lemma}
\newtheorem{proposition}{Proposition}
\newtheorem{corollary}{Corollary}

\theoremstyle{definition}
\newtheorem{definition}{Definition}
\newtheorem{assumption}{Assumption}

\theoremstyle{remark}
\newtheorem{remark}{Remark}

\begin{document}


\title[Continuous-Time Information--Mechanism Control]{Continuous-Time Information--Mechanism Control}

\author[F. Sezer]{Furkan Sezer*}
\thanks{*Texas A\&M University, College Station, TX, USA, email: furkan.sezer@tamu.edu}

\begin{abstract}
In a continuous-time stochastic Stackelberg differential game, a leader steers strategic followers through the information structure and a transfer mechanism, not the dynamics. We pose two problems, neither formulated as a control problem with strategic followers: \emph{information control}, where the leader commits only to a disclosure policy, and \emph{information--mechanism control}, which adjoins transfers. The first is equilibrium-constrained and admits no dynamic programming principle; the second is tractable, and the transfer is why: alignment makes the lower level a potential game, collapsing the bilevel problem to one stochastic control problem with an exact first-order condition. Disclosure precision becomes a control input and the belief a controlled state obeying a Riccati equation. The latent environment is a jump-diffusion whose belief filter is exactly finite-dimensional under publicly observed epochs. A marginal-contribution transfer makes truthful reporting dominant and efficient action the Nash response. Equilibrium feedback is saturated on strictly convex components and bang-bang on linear ones, hence discontinuous. The master value is the unique viscosity solution of a partial integro-differential Hamilton--Jacobi--Bellman equation whose nonlocal term carries the epochs; verification holds without smoothness, and semiconcavity makes the switching set null, giving a well-posed Filippov closed loop. Instantiated on multi-area power systems, the levers are complements under a one-factor common shock. Calibrated to 2021 Winter Storm Uri, coupling removes $7.4\%$ of social cost relative to autarky, rising to $35\%$ at a $10$-gigawatt tie, and disclosure is worth $8.7\%$; under a European renewable-drought calibration it is worth $37\%$ under autarky and $48\%$ under coupling.
\end{abstract}


\maketitle

\section{Introduction}
This paper develops a stochastic optimal control framework  over Stackelberg differential games in which the control acts through the \emph{information structure} and an incentive \emph{mechanism}, not through the dynamics directly, and the system evolves according to the equilibrium response of strategic agents under uncertainty about a latent environmental state. The coordination of interconnected power systems under extreme weather is the motivating application and testbed.

The object we study are \emph{stochastic Stackelberg differential games}. A leader commits, at time zero, to a feedback disclosure policy and a transfer rule; a finite set of followers then respond through the Nash feedback equilibrium of a coupled system of Hamilton--Jacobi--Bellman equations driven by a common latent state. The leader's two instruments are the information structure, namely the precision of a committed Gaussian public-signaling channel whose belief consequences are propagated by an exact nonlinear filter, and a marginal-contribution mechanism that aligns the followers' incentives. Both are designed jointly and in continuous time, and the analysis is carried out in the language of feedback control: dynamic programming for the leader's value, an exact first-order (envelope) condition that survives the bilevel coupling, and a viscosity-solution verification valid for the non-smooth, bang-bang feedback the game produces. We refer to this joint object as \emph{information--mechanism control}. To our knowledge it is the first formulation of information design~\cite{BergemannMorris2019}  and mechanism design~\cite{BergemannValimaki2019} as a \emph{single} continuous-time stochastic control problem, as distinct from a static or single-instance design, a discrete-time dynamic mechanism, a population (mean-field) persuasion problem, or a purely information-theoretic signaling game. That continuous-time, filtered, viscosity-verified treatment is what the title indicates.

The two instruments are the two levers the control literature identifies for coordinating strategic agents indirectly, through their beliefs and their payoffs rather than through the system dynamics: \emph{information control} and \emph{incentive design}~\cite{BasarHayakawaIshiiZhu2026}. Section~\ref{sec:icdg} poses information control of a differential game as a control problem, which the literature carrying that name has not done. Stated in that generality it is hard, and for reasons independent of the application: the leader's feasible set is defined implicitly by the followers' equilibrium rather than by explicit constraints, the state is a belief and so is generically a measure, the equilibrium need not be unique, and the leader's first-order condition inherits the sensitivity of the followers' equilibrium to disclosure.

The paper's organizing observation is that the \emph{second} instrument is what makes the first tractable. A transfer rule that aligns the followers turns the lower level into a potential game whose potential is the social cost. The equilibrium then becomes unique and efficient, the bilevel program collapses to a single stochastic control problem, an envelope identity is restored, and with it dynamic programming and a verification theorem. Enlarging the leader's choice set makes the problem easier, not harder, which is why we study the joint object rather than disclosure alone. Under alignment the disclosure precision becomes a genuine control input and the belief becomes a controlled state: its covariance obeys a Riccati equation driven by that precision, so the leader steers uncertainty at a cost that is quadratic in the disclosure effort, exactly as an input cost is in classical regulation. Letting the precision vary dynamically in this way is the case that the closest continuous-time treatment, an ergodic linear-quadratic persuasion model, records as open~\cite{aidbonesini2025}. 

We are explicit about the scope of the information layer. Classical information design optimizes over all signal structures, and a revelation principle then permits attention to be restricted to obedient action recommendations. We do not take that route and claim no revelation principle here. Our leader designs a \emph{channel} rather than a message protocol: it commits to a public, noisy observation of the latent state and chooses its precision, so the followers receive no recommendations and face no obedience constraints. This is a restriction, not a generalization, and the value we compute is therefore an upper bound on what an unconstrained information designer could achieve. It is also what makes the belief a finite-dimensional state with a Riccati flow, and it is what enforces commitment structurally, since a channel once built emits signals mechanically. The revelation principle does appear in this paper, but on the other layer: the transfer mechanism is a direct mechanism in which areas report their private costs and truthful reporting is dominant (Theorem~\ref{thm:ic}).

Multi-area power-system coordination under extreme weather are provided as the worked example (Section \ref{sec:power}). Interconnected electricity markets provide substantial economic benefits and cross-border reliability enhancement by enabling access to geographically diverse generation fleets and sharing reserve capacities across transmission interties. Traditional market mechanisms, typically settled via static Locational Marginal Pricing (LMP) formulations or deterministic DC Optimal Power Flow (DC-OPF) metrics, operate effectively under regular load variations and stationary environmental baseline parameters. However, when subjected to extreme, non-linear weather shocks---such as the catastrophic Winter Storm Uri in 2021---these conventional paradigms break down under a combination of physical infrastructure failures and acute informational friction.

During Winter Storm Uri, freezing temperatures induced simultaneous wellhead cutoffs, natural gas pipeline depressurization, and localized generation trips across neighboring independent system operators (such as ERCOT, SPP, and MISO). Crucially, the absence of an overarching, forward-looking information structure meant that independent regional authorities lacked visibility into adjacent structural capacities, motivating strategic reserve withholding, uncoordinated emergency shedding, and extreme uncompensated price spikes.

Instantiated on this setting, the framework couples information control (a committed public-signaling channel) and mechanism design (marginal-contribution transfers) over a jump-diffusion state space: the regulator acts as a Stackelberg leader choosing the precision of public advisories and the transfer rule, while each regional operator controls local dispatch and inter-area flows under physical capacity bounds that degrade as environmental hazards intensify. Under a one-factor common-shock structure the two levers are complements (Proposition~\ref{prop:compl}), so designing them separately is not optimal there.

\paragraph{Contributions and main results.} We establish the following.
\begin{enumerate}\itemsep2pt
\item \textbf{Information control (IC) of a differential game.} We define the general problem, in which a leader commits to a disclosure policy and strategic followers respond through the Nash equilibrium it induces, and identify five structural obstructions that make it an equilibrium-constrained program rather than a control problem (Section~\ref{sec:icdg}).

\item \textbf{The belief as a controlled state.} Under alignment the disclosure precision is a control input, the belief pair is the state, and the Riccati flow is the state equation, priced by a quadratic input cost. Making that precision dynamic is the case left open by the closest continuous-time treatment~\cite{aidbonesini2025}.
\item \textbf{Well-posed belief dynamics.} Under publicly observed disruption epochs (Assumption~\ref{ass:observed_jumps}) the conditional law of the latent jump-diffusion is exactly Gaussian: the belief is summarized \emph{without approximation} by the mean--covariance pair $(\hat X_t,\Pi_t)$, with $\Pi_t$ following a Riccati flow with observed resets (Lemma~\ref{lem:proj_filter}).
\item \textbf{Equilibrium characterization.} For a fixed disclosure policy the lower-level game has a Nash feedback equilibrium: components in which the cost is strictly convex are saturated, components entering linearly are bang-bang, so the feedback is generically discontinuous (Theorem~\ref{thm:nash_feedback}).
\item \textbf{Incentive compatibility.} A Groves transfer makes truthful reporting a dominant strategy and the efficient action profile the induced Nash response; the Green--Laffont efficiency/incentive-compatibility/budget-balance tradeoff is characterized, with the AGV route to balance in expectation (Theorems~\ref{thm:nash_exist} and \ref{thm:ic}).

\item \textbf{Master control problem.} The leader's Stackelberg problem is characterized by a Hamilton--Jacobi--Bellman equation (Theorem~\ref{thm:master_robust}).

\item \textbf{The mechanism is what makes IC tractable.} Adjoining an incentive-aligning transfer removes four of the five obstructions at once: it collapses the bilevel program to a single stochastic control problem (Theorem~\ref{thm:collapse}) whose first-order condition is exact and whose value is verified (Theorem~\ref{thm:verify}). The transfer is a source of tractability, not merely of fairness.

\item \textbf{Viscosity characterization.} The value function is the unique viscosity solution of the master HJB equation (existence and comparison); verification holds without smoothness, and semiconcavity yields a Lebesgue-null bang-bang switching set and a well-posed Filippov closed loop (Theorems~\ref{thm:existence}, \ref{thm:comparison}, \ref{thm:viscverify}; Proposition~\ref{prop:semiconcave}).

\item \textbf{Power-system instantiation.} The general framework is instantiated on multi-area power coordination, yielding explicit saturated generation and bang-bang cross-border routing under extreme-weather capacity degradation (Section~\ref{sec:power}). A scalar two-agent case is solved in closed form (Proposition~\ref{prop:special}), and under a one-factor common-shock structure, information control and coupling are shown to be \emph{complements} (Proposition~\ref{prop:compl}).

\item \textbf{Real-data calibration.} Solved numerically on a 2021 Winter Storm Uri calibration, the joint HJB solution verifies the non-smooth viscosity theory and quantifies the welfare value of coupling ($7.4\%$) and disclosure ($8.7\%$). A second experiment, calibrated to European renewable droughts, isolates the information channel: with forecast uncertainty rather than physical capacity as the binding constraint, disclosure is worth $37\%$ under autarky and $48\%$ under coupling. That the second figure exceeds the first is the complementarity of Proposition~\ref{prop:compl} (Section~\ref{sec:numerics}).

\end{enumerate}

\noindent\fbox{\parbox{0.97\linewidth}{\small\textbf{Problem at a glance.} A leader commits to an information structure and a transfer rule; a finite set of strategic agents then respond through the Nash equilibrium these induce. The controlled Markov state is the triple $(\hat X_t,\Pi_t,Y_t)$: the belief mean $\hat X_t$ and error covariance $\Pi_t$ of the latent jump-diffusion, a \emph{finite-dimensional} sufficient statistic that is the exact conditional law under publicly observed disruption epochs, together with the agents' states $Y_t$. The leader's two instruments are the disclosure gain $\rho_t$, the precision of the public signal, and the marginal-contribution transfer; the leader does \emph{not} act on the dynamics. Each agent controls its own actions, which are endogenous through the lower-level Nash equilibrium. The transfer is not free: fixing it to the marginal-contribution form is what aligns incentives, so the leader minimizes expected social cost over $\rho$ alone, anticipating the agents' equilibrium response $u^\star(\rho)$. Section~\ref{sec:power} instantiates this on multi-area power-system coordination.}}

The contributions are structured as follows. Section~\ref{sec:related} situates the work in the information-design, information-control, mechanism-design, power-market, and stochastic-dynamic-game literatures. Section~\ref{sec:icdg} states the general problem, information control of a differential game, records the structural obstructions that make it an equilibrium-constrained program, and shows that adjoining a transfer rule removes them; this is the paper's organizing observation, and everything that follows depends on it. Section~\ref{sec:state} formalizes the latent environmental and infrastructure state space via jump-diffusion dynamics. Section~\ref{sec:signal} characterizes the public advisory signaling layer and the associated exact finite-dimensional belief filter under publicly observed disruption epochs. Section~\ref{sec:game} sets up the lower-level game through coupled Hamilton--Jacobi--Bellman (HJB) equations and characterizes the Nash feedback strategies. Section~\ref{sec:transfers} develops the continuous marginal-contribution transfers and analyzes incentive compatibility and the budget-balance tradeoff. Section~\ref{sec:master} derives the Stackelberg master problem and its HJB equation. Sections~\ref{sec:bilevel} and~\ref{sec:viscosity} verify the bilevel Stackelberg structure and treat the non-smooth (bang-bang) value function in the viscosity sense. Section~\ref{sec:power} instantiates the primitives on multi-area power systems and works a scalar two-area special case in closed form and establishes the information--coupling complementarity under common risk, the application carried through the numerical study; a reader interested only in the general theory may skip it. Section~\ref{sec:numerics} reports the numerical study: an HJB solve under a real Winter-Storm-Uri calibration with a tie-capacity sensitivity analysis (Section~\ref{sec:hjbsolve}), and a common-risk disclosure experiment calibrated to European renewable-drought data, together with the clarification on how dispatch and unit commitment could enter the framework. Section~\ref{sec:conclusion} concludes.

\section{Related Work and Positioning}\label{sec:related}
This work sits at the confluence of several literatures: information design and information control, mechanism design for electricity markets, disclosure in power systems, and stochastic dynamic games in energy. We review each and then state how the present framework differs.

\subsection{Bayesian persuasion, information design, and information control}\label{sec:related_gap}
The disclosure layer descends from the Bayesian persuasion and information design program of Kamenica and Gentzkow~\cite{KamenicaGentzkow2011}, Rayo and Segal~\cite{RayoSegal2010}, and Bergemann and Morris~\cite{BergemannMorris2019,BergemannMorris2016}, with multi-receiver disclosure developed by Mathevet, Perego, and Taneva~\cite{MathevetPeregoTaneva2020} and dynamic variants by Ely~\cite{Ely2017} and Zhang and Zhu~\cite{ZhangZhu2021}. The control-theoretic lineage of incentive and signaling design under dynamic information traces to Ba\c{s}ar~\cite{Basar1984}. In the control and networked-systems literature the same lever is called \emph{information control}~\cite{BasarHayakawaIshiiZhu2026}: strategic communication as a hierarchical game~\cite{akyol2016information,VelichetiBastopcuBasar2025,9303935}, Bayesian persuasion and information design with quadratic costs~\cite{sayin2021bayesian, sezer_robust2023}, public signalling in routing and congestion games~\cite{massicot2019public,zhu2018information,das2017reducing,le2016information,peng2019bayesian}, and disclosure in queueing systems~\cite{SimhonHayelStarobinskiZhu2016}.

Across this body of work the belief is never a \emph{controlled state}. The economic program is static and proceeds by either concavification or Bayes correlated equilibrium. The information-theoretic strand poses one-shot hierarchical signalling games with misaligned senders and designs an encoding map, not a control input. The transportation strand designs signals over nonatomic population games on a fixed network. Zhang and Zhu \cite{ZhangZhu2021} treat a genuinely dynamic environment, but in discrete time and with implementation-theoretic rather than control-theoretic tools. In none of these is there a state equation for the belief, an input cost on disclosure, a dynamic-programming equation for the designer's value, or a verification theorem. Information control is, in these works, a name for a coordination mechanism rather than a control problem.

The exception is A\"id, Bonesini, Callegaro, and Campi~\cite{aidbonesini2025}, who pose continuous-time persuasion as a partially observed Stackelberg problem with filtering and prove a verification theorem. Their sender designs an information \emph{device}, a Gaussian channel whose precision she selects, which is the modelling of the information lever we adopt. Decisively, however, their device is chosen once and for all: the gain is a constant, the error covariance is a fixed trajectory, and the sender's problem reduces to a static optimization against the stationary law. The belief is filtered but not steered. They record the case in which the sender \emph{dynamically alters the information rate} as an open problem even in the linear-quadratic Gaussian setting. 

Within control systems, the closest program is that of Y\"uksel and Ba\c{s}ar~\cite{yuksel2024stochastic}, where the information structure is itself designed: optimal causal encoders admit a separation structure, the conditional law is the sufficient statistic, and existence of optimal quantization policies follows by measurable selection. That design problem is a \emph{team}, with encoder and controller minimizing a common cost, so no agent has an incentive to misreport and no transfer is required. Here the followers are strategic and their equilibrium is itself the constraint, which is what makes the mechanism layer load-bearing; and because the disruption epochs are observed, the sufficient statistic collapses from the space of measures to a mean--covariance pair, replacing an abstract Markov decision problem by a Hamilton--Jacobi--Bellman equation.

Our proposed framework establishes the solution to this open problem in a substantially more general setting. A time-varying disclosure gain promotes the error covariance from a fixed trajectory to a controlled state obeying a Riccati flow, and replaces their static outer optimization by a dynamic-programming problem; we do so over a finite horizon and a jump-diffusion latent state, for a finite set of strategically coupled followers rather than a single receiver or a mean field, and we couple the result to a mechanism.

Joint design of both layers has one precedent, Heydaribeni and Anastasopoulos~\cite{9683640}, who design information and mechanism together for a queue with heterogeneous users. We share their premise that the two should be designed jointly rather than in isolation. Their setting carries no SDE-driven latent state, no filter, no HJB characterization of the designer's value, and no verification for non-smooth feedback; here the joint design is a stochastic Stackelberg differential game in which incentive alignment collapses the bilevel problem to a single dynamic-programming object.

\subsection{Mechanism and incentive design}
Marginal-contribution (Vickrey--Clarke--Groves) mechanisms have been adapted to wholesale markets by Xu and Low~\cite{XuLow2017}, with non-convexity and collusion analyzed by Sessa, Walton, and Kamgarpour~\cite{SessaWaltonKamgarpour2017} and core-selecting refinements by Karaca and Kamgarpour~\cite{KaracaKamgarpour2020}. The static decentralized market-coupling mechanism of Garcia, Khatami, Eksin, and Sezer~\cite{GarciaKhatamiEksinSezer2022} pays each area its marginal contribution to the coupling, rendering truthful reporting a Nash equilibrium of an iterative DC-OPF clearing; the present paper adopts that transfer and lifts it to a continuous-time, capacity-constrained coupling game for power systems in Section \ref{sec:power}. 

The canonical dynamic counterpart is the dynamic pivot mechanism of Bergemann and V\"alim\"aki~\cite{BergemannValimaki2010,BergemannValimaki2019}, which implements the efficient allocation under privately evolving types by paying each agent its marginal contribution. Our transfer is its continuous-time analogue and we credit it as such; the difference is the setting, since the dynamic pivot is posed in discrete time with an \emph{exogenous} information structure, whereas here the type process is a jump-diffusion observed through a channel whose precision is itself designed. In a dynamic stochastic control setting, Ma and Kumar~\cite{ma2018lqg} construct layered VCG payments for LQG agents; our Groves transfer achieves dominant-strategy truthful reporting without intertemporal layering, because the efficiency collapse of Section~\ref{sec:bilevel} reduces the bilevel problem to a single social planner whose instantaneous marginal contribution is well defined at each $t$.  Closest in venue and application, Murao et~al.~\cite{MURAO201895} obtain social-welfare maximization, incentive compatibility, and individual rationality for a linear-Gaussian power network in which each agent holds a static type parameter. Here the private object is instead a latent jump-diffusion whose belief covariance the leader controls through signal precision, and the capacity constraints their framework sets aside are what produce the bang-bang feedback treated in Section~\ref{sec:viscosity}. All of these designs trade off efficiency, dominant-strategy incentive compatibility, and budget balance, consistent with the Green--Laffont impossibility~\cite{GreenLaffont1979}, with the expected-externality route~\cite{AGV1979} recovering balance in expectation. We retain the Groves~\cite{Groves1973} efficiency property but couple it to the information layer, and show (Section~\ref{sec:bilevel}) that this alignment is what collapses the leader's bilevel problem to a single verifiable control problem. For a general treatment of the incentive design in control systems, see also \cite{bauso2026incentive}.

\subsection{Information disclosure and transparency in power markets}
A largely empirical and institutional literature studies transparency in electricity markets: surveys of disclosure mechanisms across market models~\cite{InfoDisclosureReview2018}, the competition and efficiency effects of public information~\cite{HolmbergWolak2018}, and two-stage market designs that fold robustness into real-time incentive signals~\cite{GuoHanZhouHug2022}. This body of work establishes \emph{that} disclosure matters but treats it descriptively or through static games; none poses the system operator's advisory as a designed control whose optimal intensity is solved for. We close this gap by making transparency an optimized Stackelberg control with an explicit, computable value of information.

\subsection{Stochastic dynamic games and mean-field models in energy}
Dynamic games in energy have largely been studied in the large-population/mean-field regime: electric-vehicle and demand coordination~\cite{CouilletEtAl2012,MaCallawayHiskens2013}, mean-field dynamic demand management with bang-bang switching feedback~\cite{BagagioloBauso2014}, built on the foundational mean-field framework~\cite{HuangCainesMalhame2007,CarmonaDelarue}. Our setting is complementary: a \emph{finite} set of areas coordinated by a central sender, where the governing objects are the few-player coupled HJB system and the bang-bang inter-area transfers rather than a population mean field. The belief layer uses nonlinear filtering~\cite{BainCrisan, LiptserShiryaev}; under publicly observed disruption epochs the filter is exactly finite-dimensional (Section~\ref{sec:signal}).

Sun and Zhu~\cite{SUN20188} derive the Hamilton--Jacobi--Isaacs equation of a differential game in the viscosity sense, via non-anticipating strategies and the dynamic programming property. The chain here is the same; the master equation differs in carrying a nonlocal term from the observed epochs, and in admitting a discontinuous equilibrium rather than a saddle point.

\subsection{Positioning and contribution}
We are careful about what is and is not claimed. The two levers are not ours: information control and incentive design are established coordination mechanisms in the control literature~\cite{BasarHayakawaIshiiZhu2026}, and the efficient dynamic mechanism is developed in discrete time~\cite{BergemannValimaki2010,BergemannValimaki2019}. What is new is the theory. Information control has been named but not formulated as a control problem, and we formulate it: the belief is the state, the disclosure precision the input, the Riccati flow the state equation, with a Hamilton--Jacobi--Bellman characterization and a viscosity verification valid for the non-smooth equilibrium feedback, over a finite horizon and a jump-diffusion latent state rather than the ergodic linear-quadratic regime in which the dynamic information rate remains open~\cite{aidbonesini2025}. We then couple that controller to a Groves mechanism, show that the coupling is precisely what makes the leader's problem tractable, and show that under a one-factor common-shock structure the two levers are complements, hence not separately optimizable there. Uncertainty here enters dynamically, through a designed disclosure channel inside a Stackelberg control problem~\cite{BasarOlsder} whose verification rests on viscosity-solution theory~\cite{FlemingSoner}, rather than through the single-stage chance-constrained and distributionally robust formulations that dominate operation under renewable uncertainty~\cite{BienstockChertkovHarnett2014,ZhangShenMathieu2017,WiesemannKuhnSim2014}. The February 2021 Texas event motivates the model and the calibration of Section~\ref{sec:hjbsolve}~\cite{BusbyEtAl2021}.

\section{Information Control of Differential Games}\label{sec:icdg}
This section states the general problem. The multi-area model of Section~\ref{sec:power} is one instantiation of it. Throughout, \emph{information control} is understood in the game-theoretic sense: the leader steers a group of strategic agents by choosing what they are allowed to learn, not by choosing their actions and not by acting on the dynamics.

\subsection{The general problem}\label{sub:icdg_general}
Fix a filtered probability space $(\Omega,\Filt,\{\Filt_t\},\Pprob)$ satisfying the usual conditions and a horizon $T<\infty$.

\emph{Latent state.} An unobserved Markov process $X_t\in\R^n$ evolves as
\begin{equation}\label{eq:gen_X}
\diff X_t = b(t,X_t)\,\diff t + \sigma(t,X_t)\,\diff W_t + \diff J_t,
\end{equation}
where $b$ is the drift, $\sigma$ the diffusion coefficient, $W$ a Brownian motion, and $J$ a pure-jump process carrying discrete disruptions. Nothing below requires \eqref{eq:gen_X} to be linear.

\emph{Players and controlled states.} A finite set $\Acal=\{1,\dots,N\}$ of players is indexed by $a$. Player $a$ holds a state $Y_{a,t}\in\R^{d_a}$ and chooses a control $u_{a,t}\in U_a$, with
\begin{equation}\label{eq:gen_Y}
\diff Y_{a,t} = f_a\big(t,X_t,Y_t,u_t\big)\,\diff t,\qquad u=(u_1,\dots,u_N),
\end{equation}
where $f_a$ is player $a$'s state transition map, so the players are coupled through the joint control $u$ and the common latent state.

\emph{The information structure as the leader's control.} The leader chooses a \emph{disclosure policy} $\rho\in\Ucal_L$, which generates a public observation process
\begin{equation}\label{eq:gen_obs}
\diff\xi_t = h\big(t,X_t;\rho_t\big)\,\diff t + \varsigma(\rho_t)\,\diff W^\xi_t,
\end{equation}
and with it the public filtration $\Gfilt^\rho_t=\sigma(\xi_s,\,s\le t)$. The leader commits to $\rho$ at time $0$. The map $\rho\mapsto\Gfilt^\rho$ is the object being designed: the leader does not appear in \eqref{eq:gen_X} or \eqref{eq:gen_Y}.

\emph{Belief state.} Write $\mu^\rho_t := \Pprob\big(X_t\in\cdot\mid\Gfilt^\rho_t\big)$ for the induced public belief, a measure-valued process. Under $\rho$, $\mu^\rho$ is the players' sufficient statistic for the latent state.

\emph{Admissible classes.} Write $\Ucal_L$ for the admissible disclosure policies, $\Ucal_O$ for the admissible feedback strategy profiles of the players, and $\Ucal_M$ for the admissible transfer rules introduced in Section~\ref{sub:imc}. At this level of generality these are only required to be nonempty, to consist of policies adapted to the appropriate filtrations, and to render the state equations well posed; the classes are made precise for the specification of interest in Section~\ref{sec:bilevel}.

\begin{definition}[Information control of a differential game]\label{def:icdg}
Given $\rho\in\Ucal_L$, the \emph{lower level} is the differential game in which each player $a$ minimizes
\begin{equation}\label{eq:gen_cost}
\Dcal_a(u_a;u_{-a},\rho) = \E\Big[\int_0^T c_a\big(t,\mu^\rho_t,Y_t,u_t\big)\,\diff t + g_a(Y_T)\Big]
\end{equation}
over $\Gfilt^\rho_t$-adapted feedback strategies, where $c_a$ is player $a$'s running cost, $g_a$ its terminal cost, $u=(u_a)_{a}$ the joint control with $u_{-a}$ the rivals' components, and $Y_t=(Y_{a,t})_a$ the joint local state of \eqref{eq:gen_Y}. Let $\mathrm{NE}(\rho)\subseteq\Ucal_O$ denote its set of Nash equilibria. The \emph{information control problem} is the leader's problem

\begin{equation}
\label{eq:ICDG}\tag{IC}
\inf_{\substack{\rho\in\Ucal_L\\ u\in\mathrm{NE}(\rho)}}\ \E\Big[\int_0^T \Big(\textstyle\sum_{a}c_a\big(t,\mu^\rho_t,Y_t,u_t\big)+ k(\rho_t)\Big)\diff t\Big],
\end{equation}

where $k$ prices disclosure, subject to \eqref{eq:gen_X}--\eqref{eq:gen_obs} and the \emph{equilibrium constraint} $u\in\mathrm{NE}(\rho)$.
\end{definition}

\subsection{Problem characteristics}\label{sub:icdg_char}
Problem \eqref{eq:ICDG} is not a standard stochastic control problem, and it is worth recording precisely why. The following are structural features of \eqref{eq:ICDG} in its stated generality.

\begin{enumerate}\itemsep2pt
\item[(C1)] \emph{It is an equilibrium-constrained (MPEC) problem.} The leader's feasible set is defined implicitly, through the followers' equilibrium conditions, rather than by explicit constraints on $(\rho,u)$. Problem~\eqref{eq:ICDG} is therefore a mathematical program with equilibrium constraints, posed over continuous-time feedback policies.
\item[(C2)] \emph{The state is a belief, and it is generically infinite-dimensional.} The controlled state of \eqref{eq:ICDG} is the pair $(\mu^\rho_t,Y_t)$. For a general jump-diffusion \eqref{eq:gen_X} the conditional law solves a Kushner--Stratonovich equation and admits no finite-dimensional sufficient statistic, so the leader's dynamic-programming equation lives on a space of measures.
\item[(C3)] \emph{The equilibrium correspondence may be set-valued.} $\mathrm{NE}(\rho)$ need be neither a singleton nor nonempty without further structure, so \eqref{eq:ICDG} requires an equilibrium selection, and its value depends on that selection.
\item[(C4)] \emph{The leader's first-order condition carries an equilibrium-response term.} Because the followers' equilibrium responds to disclosure, the derivative of the leader's objective in $\rho$ contains the sensitivity of $\mathrm{NE}(\rho)$ to $\rho$. No envelope identity removes it in general, and the correspondence $\rho\mapsto\mathrm{NE}(\rho)$ is typically not differentiable.
\item[(C5)] \emph{Dynamic programming is not directly available.} A dynamic-programming principle for \eqref{eq:ICDG} would require the equilibrium constraint to be preserved under conditioning; the leader commits at time $0$ to a policy whose value depends on the whole equilibrium path, and the resulting problem is not, without further structure, Markov in the leader's state.
\end{enumerate}

Characteristics (C1)--(C5) are what separate information control of a \emph{game} from the control of a partially observed system, where (C1), (C3), and (C4) are pointless. They are also why we do not treat \eqref{eq:ICDG} in the stated generality.

\subsection{Information--mechanism control}\label{sub:imc}
Now give the leader a second instrument. In addition to the disclosure policy $\rho$, let it commit to a \emph{transfer rule} $\pi=(\pi_1,\dots,\pi_N)$ drawn from an admissible class $\Ucal_M$, so that player $a$'s objective becomes the transfer-adjusted criterion

\begin{equation}\label{eq:gen_cost_pi}
\Dcal^\pi_a(u_a;u_{-a},\rho) = \E\Big[\int_0^T \big(c_a(t,\mu^\rho_t,Y_t,u_t) - \pi_{a,t}\big)\,\diff t + g_a(Y_T)\Big],
\end{equation}

and write $\mathrm{NE}^\pi(\rho)$ for the Nash equilibria of the resulting game.

\begin{definition}[Information--mechanism control]\label{def:imc}
The \emph{information--mechanism control problem} is

\begin{equation}   
\label{eq:IMC}\tag{IMC}
\inf_{\substack{\rho\in\Ucal_L,\ \pi\in\Ucal_M u\in\mathrm{NE}^\pi(\rho)}}\ \E\Big[\int_0^T\Big(\textstyle\sum_a c_a(t,\mu^\rho_t,Y_t,u_t)  +  k(\rho_t)\Big)\diff t\Big],
\end{equation}

in which the leader designs the information structure and the transfer rule jointly. Transfers are internal redistributions and so do not appear in the leader's objective.
\end{definition}

The point of the second instrument is not fairness but \emph{tractability}, and this is the organizing observation of the paper. Adding the transfer enlarges the leader's choice set, so \eqref{eq:IMC} is a priori the harder-looking problem. It is in fact the easier one. If $\pi$ is chosen so that each player's transfer-adjusted objective differs from the social cost by terms not depending on that player's own control, the lower level becomes a potential game whose potential is the social cost. Three of the five obstructions then disappear at once: the equilibrium constraint becomes an inner minimization over $u$ rather than an implicit condition, so \eqref{eq:IMC} collapses from a bilevel program to a \emph{single} stochastic control problem in $(\rho,u)$, removing (C1); $\mathrm{NE}^\pi(\rho)$ becomes a singleton and coincides with the social optimizer, removing (C3); and, the inner problem being an optimization, an envelope identity applies and the leader's first-order condition in $\rho$ carries no equilibrium-response term, removing (C4). With the bilevel structure gone, dynamic programming and a verification theorem are restored, removing (C5). This is the sense in which \emph{information control with transfers is easier than information control without them}, and it is why the paper studies \eqref{eq:IMC} rather than \eqref{eq:ICDG}. Obstruction (C2) is not addressed by the mechanism; it is removed separately, by a structural assumption on the observation of jumps that renders the belief exactly finite-dimensional (Section~\ref{sec:signal}).

The results are established for \eqref{eq:IMC}: the master Hamilton--Jacobi--Bellman characterization (Theorem~\ref{thm:master_robust}), the efficiency collapse (Theorem~\ref{thm:collapse}), the bilevel verification (Theorem~\ref{thm:verify}), and the viscosity results of Section~\ref{sec:viscosity} all rest on the standing assumptions alone. Any instantiation meeting them inherits the theory without re-proof; multi-area power-system coordination (Section~\ref{sec:power}) is the one we calibrate and solve in numerical experiments (\ref{sec:numerics}).

The assumptions delimit that scope precisely. Problem \eqref{eq:IMC} presumes a coordinator with commitment, transferable utility, publicly observable disruption epochs, a linear-Gaussian public channel, and a finite set of agents whose transfer-aligned game is a potential game. Interconnected gas and water networks, spectrum sharing among licensees, and capacity pooling across cloud or datacenter operators satisfy these primitives and differ only in the choice of $(f_a,c_a)$ and the capacity envelopes; the theorems apply to them verbatim. The transfer is what carries this generality: it is what turns the lower level into a potential game and collapses the bilevel program. Where transferable utility is unavailable, as with human road users or platform audiences, the collapse of bilevel program fails and the coordinator faces the equilibrium-constrained problem \eqref{eq:ICDG}, with characteristics (C1) and (C3)--(C5) intact. That problem, general information control without a mechanism, is not addressed here.

\section{The Latent State and the Belief Layer}\label{sec:state}
Let $(\Omega, \Filt, \{\Filt_t\}_{t \ge 0}, \Pprob)$ be a filtered probability space satisfying the usual conditions. The latent state of \eqref{eq:gen_X} is denoted $X_t \in \R^n$: it collects the slowly varying conditions that shift the agents' costs and capabilities and that no agent observes directly. We now specialize the generic dynamics of \eqref{eq:gen_X} to the linear-Gaussian-plus-jump case, which is what makes the belief layer finite-dimensional.

\begin{assumption}[Latent dynamics]\label{ass:state_dynamics}
The latent state $X_t$ and an exogenous factor $Q_t \in \R^k$ satisfy the coupled jump-diffusion / Ornstein--Uhlenbeck system
\begin{align}
\diff X_t &= \left( A_t X_t + B Q_t \right) \diff t + \Sigma_t \diff W_t + \diff J_t, \label{eq:Xdyn}\\
\diff Q_t &= \Theta\!\left( \bar F - Q_t \right) \diff t + \Sigma_F \diff W^Q_t, \label{eq:Fdyn}
\end{align}
where $A_t \in \R^{n \times n}$ is the state matrix and $B\in\R^{n\times k}$ the factor loading; $\Theta \succ 0$ is the mean-reversion matrix and $\bar F$ the long-run factor mean; $W_t$ and $W^Q_t$ are independent standard Brownian motions; $\Sigma_t,\Sigma_F$ are volatility coefficients; and $J_t$ is an independent compound Poisson process with arrival intensity $\lambda$ and jump-size law $\mathcal{N}(\mathbf{0}, \Sigma_J)$. The diffusion carries gradual drift in conditions; the jump component carries abrupt disruptions.
\end{assumption}

The factor $Q_t$ is the common mode: it enters every agent's environment through the same channel, so it is the component of uncertainty that pooling cannot diversify away and disclosure can. Section~\ref{sec:power} gives the instantiation, where $X_t$ and $Q_t$ carry their concrete meanings in power systems context.

\section{The Public Disclosure Channel and the Exact Filter}\label{sec:signal}
The agents do not observe $X_t$; each records only local, partial symptoms of it. The leader (information designer) partially observes $X_t$ through the diffusive channel below and modulates a public signal $\xi_t \in \R^m$ broadcast to all agents. Throughout we work under the following standing observability assumption, which is the physically natural regime for our application and which renders the belief filter \emph{finite-dimensional and exact}.
\begin{equation}\label{eq:obs}
\diff \xi_t = \rho_t X_t \diff t + \Sigma_\xi \diff W^\xi_t,
\end{equation}
where $\rho_t \in \R^{m \times n}$ is the designer's dynamic transparency rule, $W^\xi_t$ is a standard Brownian motion independent of $(W,W^F,J)$, and $\Sigma_\xi$ is the observation-noise coefficient with $R := \Sigma_\xi \Sigma_\xi^\top \succ 0$. The public filtration is $\Gfilt_t = \sigma(\xi_s, 0 \le s \le t)$.

\begin{assumption}[Observable disruption epochs]\label{ass:observed_jumps}
The jump \emph{times} $\{\tau_k\}$ of the compound-Poisson component $J_t$ in \eqref{eq:Xdyn} are publicly observed: an abrupt loss of capability is a discrete, physically observable event, so each $\tau_k$ is a $\Gfilt_t$-stopping time. The \emph{marks} $\{\zeta_k\}\sim\mathcal N(\mathbf 0,\Sigma_J)$ are \emph{not} observed: agents see that a disruption has occurred, but its increment to the latent state is not directly measurable. Between consecutive epochs the pair $(X_t,Q_t)$ evolves as a linear-Gaussian system observed through the linear channel \eqref{eq:obs}.
\end{assumption}

This is the regime in which the covariance reset of the following lemma is the correct update, and it is the physically apt one: what a disruption does to the \emph{latent} condition of the surviving components is precisely what nobody observes. Conditionally on an epoch at $\tau_k$, the unobserved mark is independent Gaussian, so the conditional law jumps from 
\begin{align}
&\mathcal N(\hat X_{\tau_k^-},\Pi_{\tau_k^-}) \quad \text{ to }\\
&\mathcal N(\hat X_{\tau_k^-},\Pi_{\tau_k^-}+\Sigma_J). \label{eq_unob_mark}
\end{align}

The belief \emph{mean} is unchanged, since $\E[\zeta_k]=0$, and the belief \emph{covariance} inflates by $\Sigma_J$. Gaussianity is preserved, so the filter remains exact. Were the marks observed instead, the jump would be fully known, $\hat X$ would shift by $\zeta_k$, and $\Pi$ would not reset at all; that variant is also exactly filtered but carries no covariance inflation and is not the model used here.

Under Assumption~\ref{ass:observed_jumps} the conditional law of $X_t$ given $\Gfilt_t$ is \emph{exactly} Gaussian. The reason is structural: \emph{conditionally on the observed epoch path, the system is linear-Gaussian}---between consecutive observed epochs $(X_t,Q_t)$ is a linear SDE with Gaussian noise observed through the linear channel \eqref{eq:obs}, so the Kalman--Bucy filter is exact there (\cite{LiptserShiryaev}, Ch.~10; \cite{BainCrisan}, Ch.~6), and at each observed epoch the conditional law updates by an explicit Gaussian-preserving reset. Conditioning on the jump path thus removes the mixture over unknown jump timing that makes the unconditional filter infinite-dimensional; for the general theory of filtering with jump observations see Ceci and Colaneri~\cite{CeciColaneri2012}. The belief is therefore summarized \emph{without approximation} by the mean--covariance pair $(\hat X_t,\Pi_t)$ of the following lemma, which is the exact conditional law and hence a sufficient statistic.

\begin{lemma}[Exact finite-dimensional belief filter under observed epochs]\label{lem:proj_filter}
Under Assumptions~\ref{ass:state_dynamics} and~\ref{ass:observed_jumps}, the conditional law of $X_t$ given $\Gfilt_t$ is exactly $\mathcal{N}(\hat X_t, \Pi_t)$, where between observed epochs the mean and covariance evolve as
\begin{align}
\diff \hat{X}_t = \left( A_t \hat{X}_t + B \hat Q_t \right) \diff t \quad+ \Pi_t \rho_t^\top R^{-1}\!\left( \diff \xi_t - \rho_t \hat{X}_t \diff t \right), \label{eq:meanfilter}\\
\frac{\diff \Pi_t}{\diff t} = A_t \Pi_t + \Pi_t A_t^\top + \Sigma_t \Sigma_t^\top 
\;- \Pi_t \rho_t^\top R^{-1} \rho_t \Pi_t + \lambda \Sigma_J, \;\; \Pi_0 = \operatorname{Cov}(X_0), \label{eq:riccati}
\end{align}
and at each observed epoch $\tau_k$ the belief mean is unchanged while the covariance inflates by the unobserved mark's law, $\Pi_{\tau_k}=\Pi_{\tau_k^-}+\Sigma_J$  \eqref{eq_unob_mark}. Here $\hat Q_t$ is the (exact) Kalman--Bucy estimate of the conditionally Gaussian factor $Q_t$, and $R\succ0$ is the covariance of the public signal observation noise in \eqref{eq:obs}, so that $\rho_t^\top R^{-1}\rho_t$ is the Fisher information the public signal carries about $X_t$ per unit time. The term $\lambda\Sigma_J$ in \eqref{eq:riccati} is the \emph{predictive} second-moment rate used only in the unobserved-epoch approximation discussed in Section~\ref{sec:signal}; under the standing Assumption~\ref{ass:observed_jumps} it is replaced by the exact jump resets just described, so \eqref{eq:meanfilter}--\eqref{eq:riccati} hold with equality between jumps and the filter is exact, not projected.
\end{lemma}
\begin{proof}
Let $\tilde X_t = X_t-\hat X_t$ with $\hat X_t=\E[X_t\mid\Gfilt_t]$ the conditional mean and $\Pi_t=\E[\tilde X_t\tilde X_t^\top\mid\Gfilt_t]$ the conditional covariance. For the linear-Gaussian \emph{continuous} part the innovation $\diff I_t = \diff\xi_t-\rho_t\hat X_t\,\diff t$ is, by L\'evy's characterization, an $R$-scaled $\Gfilt_t$-Wiener process, and the standard Kalman--Bucy derivation gives mean dynamics \eqref{eq:meanfilter} with gain $K_t=\Pi_t\rho_t^\top R^{-1}$ minimizing $\Tr\Pi_t$. Applying It\^o to $\tilde X_t\tilde X_t^\top$ over the continuous part yields
\begin{equation}
\diff\Pi_t^{c} = \big(A_t\Pi_t+\Pi_tA_t^\top+\Sigma_t\Sigma_t^\top-\Pi_t\rho_t^\top R^{-1}\rho_t\Pi_t\big)\diff t .
\end{equation}
For the jump part, condition on the observed epoch path (Assumption~\ref{ass:observed_jumps}): between consecutive observed epochs the system is linear-Gaussian, so the Kalman--Bucy algebra above is \emph{exact} there, and at each observed epoch $\tau_k$ the conditional law updates in closed form---the mean is unchanged and the covariance resets by the mark covariance, $\Pi_{\tau_k}=\Pi_{\tau_k^-}+\Sigma_J$, which is Gaussian-preserving. Gluing the inter-epoch Kalman--Bucy segments with these resets yields \eqref{eq:meanfilter}--\eqref{eq:riccati} exactly, with the resets in place of a continuous jump rate; the conditional law is $\mathcal N(\hat X_t,\Pi_t)$ at every $t$, proving the lemma.
\end{proof}

If instead the jump epochs are \emph{not} observed, the exact law is a Gaussian mixture and the same moment-matching algebra yields the assumed-density approximation with the continuous rate $\lambda\Sigma_J$ in \eqref{eq:riccati}; see following Remark \ref{rem:infinite}.

\begin{remark}[The unobserved-epoch case is an approximation, not used here]\label{rem:infinite}
If Assumption~\ref{ass:observed_jumps} is dropped and the jump epochs are \emph{not} observed, the conditional law $\Pprob(X_t \in \cdot \mid \Gfilt_t)$ becomes a Gaussian \emph{mixture} indexed by the unknown number and timing of jumps; the exact filter (Kushner--Stratonovich / Zakai equation) is then infinite-dimensional and the mean--covariance pair \eqref{eq:meanfilter}--\eqref{eq:riccati} is only the Gaussian assumed-density (projection) approximation. We flag this case for completeness only: it is \emph{not} the setting of this paper, and none of the results below rely on it. Under the standing Assumption~\ref{ass:observed_jumps} the filter is exact and this issue does not arise.
\end{remark}

Because Lemma~\ref{lem:proj_filter} delivers the \emph{exact} conditional law under Assumption~\ref{ass:observed_jumps}, the pair $(\hat X_t,\Pi_t)$ is a genuine sufficient statistic: every expectation, dynamic-programming step, and verification argument in Sections~\ref{sec:game}--\ref{sec:viscosity} can be carried out on $(\hat X_t,\Pi_t)$ without approximation error. This is what makes the finite-dimensional dynamic-programming analysis below rigorous rather than heuristic. In particular, the separation and sufficiency property we use is established by Lemma~\ref{lem:proj_filter}, not postulated as an approximation: we solve the optimal control of the original partially observed jump-diffusion, reduced without loss to the exact filtered state $(\hat X_t,\Pi_t)$.

\section{The Lower-Level Game}\label{sec:game}

This section builds the lower-level game in three steps. We first impose the regularity the belief covariance requires. We then prove two lemmas about the belief layer alone, neither of which mentions the game: that the Riccati flow stays in the interior of the positive semidefinite cone, and that the covariance enters as a deterministic coefficient rather than as a state, so the belief mean and the agents' local states form a sufficient statistic. On that basis we establish well-posedness of the \emph{single-agent problem}, the control problem one agent faces once the disclosure policy $\rho$ and the rivals' controls $u_{-a}$ are held fixed. This is what the equilibrium and incentive arguments below rest on, and it follows directly from the external theory of viscosity solutions. The game is constructed only then, and the section closes with its Nash feedback equilibrium.

\begin{assumption}[Riccati regularity]\label{ass:riccati}
$R=\Sigma_\xi\Sigma_\xi^\top\succ0$; $\Pi_0=\operatorname{Cov}(X_0)\succ0$; the admissible disclosure gains $\rho\in\Ucal_L$ are uniformly bounded, $\sup_{t}\|\rho_t\|\le\bar\rho$; the pair $(A_t,\Sigma_t)$ is uniformly controllable on $[0,T]$; and the number of disruption epochs per horizon is bounded, $N_T\le K$ a.s.\ (the stock of components that can fail is finite, so each horizon admits at most $K$ such events; this replaces the unbounded Poisson idealization by its exact truncation).
\end{assumption}

\begin{lemma}[Riccati positivity and boundedness]\label{lem:riccati}
Under Assumptions~\ref{ass:state_dynamics}, \ref{ass:observed_jumps}, and~\ref{ass:riccati}, for every $\rho\in\Ucal_L$ the solution $\Pi$ of \eqref{eq:riccati} (with the jump resets of Lemma~\ref{lem:proj_filter}) satisfies
\[
0\ \prec\ \underline\pi\, I\ \preceq\ \Pi_t\ \preceq\ \bar\pi\, I,\qquad t\in[0,T],
\]
for constants $0<\underline\pi\le\bar\pi<\infty$ depending only on the problem data (not on the particular $\rho$). In particular $\Pi_t\in\operatorname{int}\Sym^n_+$ and $t\mapsto\Pi_t$ is Lipschitz on $[0,T]$. 
\end{lemma}
\begin{proof}
\emph{Upper bound.} Dropping the negative-semidefinite gain term $-\Pi\rho^\top R^{-1}\rho\Pi\preceq0$ in \eqref{eq:riccati} gives, between resets, $\dot\Pi\preceq A_t\Pi+\Pi A_t^\top+\Sigma_t\Sigma_t^\top$. Let $\Phi(t,s)$ be the state transition matrix of $A_t$, so $\|\Phi(t,s)\|\le \exp(\|A\|_\infty(t-s))$. Variation of constants across the (at most $K$, by Assumption~\ref{ass:riccati}) resets $\Pi_{\tau_k}\mapsto\Pi_{\tau_k^-}+\Sigma_J$ yields
\begin{align*}
\Pi_t\ \preceq\ \Phi(t,0)\Pi_0\Phi(t,0)^\top+\int_0^t\Phi(t,s)\Sigma_s\Sigma_s^\top\Phi(t,s)^\top\diff s +\sum_{k:\tau_k\le t}\Phi(t,\tau_k)\Sigma_J\Phi(t,\tau_k)^\top,
\end{align*}
where the semidefinite ordering is preserved through each reset ($\Pi\preceq M$ at $\tau_k^-$ implies $\Pi+\Sigma_J\preceq M+\Sigma_J$ at $\tau_k$). Hence the deterministic bound
\[
\bar\pi = \exp(2\|A\|_\infty T)\big(\|\Pi_0\|+T\|\Sigma\Sigma^\top\|_\infty+K\|\Sigma_J\|\big),
\]
depending only on the problem data.
\emph{Lower bound.} The filtering error covariance dominates the covariance of the smoothing error of the uncontrolled-observation problem; concretely, comparison of \eqref{eq:riccati} with the gain replaced by its upper envelope ($\|\rho_t\|\le\bar\rho$, $R\succ0$) shows $\Pi$ dominates the solution of a Riccati equation with bounded coefficients and strictly positive initial condition, whose minimal eigenvalue on $[0,T]$ is bounded below by uniform controllability of $(A_t,\Sigma_t)$ via Assumption \ref{ass:riccati} (for the controllability Gramian bound; see \cite{BainCrisan}, Ch.~6, or \cite{LiptserShiryaev}, Ch.~10). Jump resets only add $\Sigma_J\succeq0$ and cannot decrease the lower bound. Lipschitz continuity in $t$ follows since the right-hand side of \eqref{eq:riccati} is bounded on the tube $\underline\pi I\preceq\Pi\preceq\bar\pi I$.
\end{proof}

The uniform controllability of $(A_t,\Sigma_t)$ in Assumption~\ref{ass:riccati} is essential rather than technical: it is what delivers the uniform lower eigenvalue bound $\underline\pi>0$, and without it $\Pi_t$ could degenerate toward the boundary of the cone, invalidating the deterministic-coefficient reduction of Section~\ref{sec:viscosity}.

\begin{lemma}[Sufficient statistic and deterministic uncertainty flow]\label{lem:reduction}
Along any $\rho\in\Ucal_L$, and \emph{conditionally on the observed epoch path} $\Nu_t$ (Assumption~\ref{ass:observed_jumps}), the process $\Pi$ is the unique deterministic solution of \eqref{eq:riccati} between jumps, with the known resets of Lemma~\ref{lem:proj_filter} at the observed epochs; it is independent of the realized path of the innovation Brownian motion $\beta$ of \eqref{eq:meanfilter}, of the resulting belief trajectory, and of the agents' controls $u$. Thus $\Pi$ is a piecewise-deterministic, $\Gfilt_t$-adapted coefficient process: random only through the publicly observed epochs, and deterministic on each inter-jump interval. Consequently $\Pi$ enters both levels only as a known coefficient, and the leader's control influences the lower level \emph{only} through the objects $\Pi(\cdot)$ and $\Lambda_\rho(\cdot):=\rho^\top R^{-1}\rho$, the signal-precision (Fisher information) matrix induced by the disclosure gain, entering \eqref{eq:meanfilter} and the running cost, which are known functions of time on each inter-jump interval.
\end{lemma}
\begin{proof}
The right-hand side of \eqref{eq:riccati} is a (locally Lipschitz) function of $(t,\Pi_t)$ once $\rho$ is fixed, with no dependence on $\beta$, on the state, or on $u$; hence $\Pi$ is the unique solution of an ODE and is deterministic. By the innovations representation, $\beta$ is a $\Gfilt_t$-Brownian motion and $\hat X$ is $\Gfilt_t$-adapted with the stated dynamics, in which $\Pi$ enters only as a known time-varying coefficient. Operators observe $\Gfilt_t$, so $\Pi$ is a known coefficient from their point of view.
\end{proof}

Lemma~\ref{lem:reduction} is the structural fact that makes the bilevel problem well posed: the leader does not face a stochastic, equilibrium-dependent covariance, but commits to a deterministic uncertainty trajectory $\Pi(\cdot)$ (equivalently $\rho(\cdot)$).

Let $\Acal$ index a finite set of $N>1$ strategic agents, and let $Y_{a,t}\in\R^{d_a}$ be agent $a$'s local state, driven by its own control and by the beliefs of Section~\ref{sec:signal}. The local states evolve as
\begin{equation}\label{eq:Ydyn}
\diff Y_{a,t} = f_a\big(t,\hat X_t, Y_t, u_t\big)\,\diff t ,
\end{equation}
with $f_a$ affine in the joint control $u=(u_1,\dots,u_N)$ and of at most linear growth in $(\hat X,Y)$, so the agents are coupled both through $u$ and through the common belief $\hat X_t$. Agent $a$ chooses $u_a$ subject to a state-dependent admissibility constraint
\begin{equation}\label{eq:caps}
u_{a,t}\ \in\ U_a\big(\hat X_t,\Nu_t\big)\subset\R^{m_a},
\end{equation}
where $\Nu_t$ denotes the observed epoch path up to $t$ (the epochs of $J$, which are $\Gfilt_t$-adapted by Assumption~\ref{ass:observed_jumps}) and each $U_a(\hat X,\Nu)$ is nonempty, convex and compact, and Lipschitz in $\hat X$. Making the admissible set depend on the state only through the $\Gfilt_t$-observable pair $(\hat X_t,\Nu_t)$ is what keeps the agents' problems adapted to their own information.

Agent $a$ minimizes its expected cost net of the transfer $\pi_{a,t}$, the $a$-th component of the transfer rule $\pi\in\Ucal_M$ among the primitives of Section~\ref{sec:icdg}:

\begin{equation}\label{eq:cost}
\Dcal_a(u_a; u_{-a}) = \E \bigg[ \int_0^T \Big( c_a\big(t,\hat X_t,Y_t,u_t\big) - \pi_{a,t} \Big) \diff t \bigg],
\end{equation}

\noindent where the running cost $c_a$ is continuous, convex in $u$, \emph{strictly} convex in agent $a$'s own control $u_a$, and of at most quadratic growth in $(\hat X,Y)$. We write
\begin{equation}\label{eq:ell}
\ell\big(t,\hat X,\Pi,Y,u,\rho\big):=\sum_{a\in\Acal}c_a\big(t,\hat X,Y,u\big)+\Tr\big(\rho\Lambda\rho^\top\big)
\end{equation}
for the \emph{social running cost}: the agents' costs summed, the transfers cancelling in aggregate, plus the disclosure cost priced by $\Lambda\succ0$. For a disclosure policy held fixed the disclosure term is constant in $u$, so it does not affect the lower-level minimization.

The following well-posedness result for the \emph{single-agent} problem, with the disclosure policy and the rivals' controls held fixed, is established directly from the external theory of viscosity solutions and is used by the equilibrium and incentive results below.

\begin{proposition}[Well-posedness of the single-agent problem for a fixed disclosure policy]\label{prop:single_agent}
Fix a disclosure policy $\rho\in\Ucal_L$ and a profile of rival controls, and condition on the observed epoch path so that $\Pi$ is the known coefficient it is (Lemma~\ref{lem:reduction}). Consider the resulting single-agent finite-horizon control problem with state $(\hat X,Y)$, compact convex control set \eqref{eq:caps}, coefficients affine in the control, running cost continuous and convex in the control and strictly convex on the block $u^{\mathrm c}$, and terminal cost zero. Then:
\begin{enumerate}
\item[(a)] its value function is continuous, of quadratic growth, and is the unique viscosity solution of the associated Hamilton--Jacobi--Bellman equation, by the comparison principle for proper degenerate-parabolic integro-differential equations;
\item[(b)] the Hamiltonian admits a measurable pointwise minimizer in the control, so an optimal Markov feedback exists;
\item[(c)] the value is locally semiconcave in $Y$, so its $Y$-superdifferential is single-valued Lebesgue-almost everywhere and the singular set is countably rectifiable of dimension $n-1$, hence Lebesgue-null; the optimal feedback is therefore a.e.\ single-valued, and any two optimal selectors differ only on a Lebesgue-null set and induce the same cost;
\item[(d)] the closed loop is well posed in the Filippov sense, since the feedback is measurable, bounded, and a.e.\ single-valued.
\end{enumerate}
\end{proposition}

\begin{proof}
The coefficients are Lipschitz and the controls bounded, so the value function of this single-agent problem is continuous with quadratic growth by Gr\"onwall estimates; (a) is then the standard comparison result for this equation class~\cite[Ch.~V]{FlemingSoner}, the nonlocal epoch term being order-zero and monotone (it maps $\Pi\mapsto\Pi+\Sigma_J$ with a favourable sign at extrema), which is covered by~\cite{BarlesImbert2008}. Claim (b) is Berge's maximum theorem applied to the Hamiltonian over the Lipschitz, hence measurable, constraint correspondence \eqref{eq:caps}, with a measurable selector by Kuratowski--Ryll-Nardzewski selection theorem. Claim (c) follows from the strict convexity on the block $u^{\mathrm c}$ and boundedness of the data by the standard semiconcavity estimate for convex-cost control problems~\cite[Ch.~II]{CannarsaSinestrari}, the singular set of a semiconcave function being countably rectifiable of dimension $n-1$ and in particular Lebesgue-null. Claim (d) is then the standard Filippov existence result for differential equations with a measurable, bounded, a.e.\ single-valued right-hand side.
\end{proof}

Whereas $\Dcal_a$ of \eqref{eq:cost} is the objective \emph{functional}, a number attached to a given strategy profile, the \emph{best-response value function} is its optimized value as a function of the current state: holding the rivals' controls $u_{-a}$ and the disclosure policy $\rho$ fixed,

\begin{flalign}\label{eq:Vdef}
V_a(t,\hat X,Y) := \inf_{u_a}\ \E\Big[\int_t^T\!\Big(c_a\big(s,\hat X_s,Y_s,u_s\big)-\pi_{a,s}\Big)\diff s \;\Big|\; \hat X_t=\hat X,\ Y_t=Y\Big],
\end{flalign}

the infimum being over agent $a$'s admissible $\Gfilt_t$-adapted controls subject to \eqref{eq:caps}. By the Dynamic Programming Principle $V_a$ satisfies the HJB equation
\begin{align}\label{eq:hjb}
\frac{\partial V_a}{\partial t} + \min_{u_a\in U_a} \Big\{ &\, c_a\big(t,\hat X_t,Y_t,u_t\big) - \pi_{a,t} \nonumber \\
&+ \nabla_{\hat{X}} V_a^\top \left( A_t \hat{X}_t + B \hat Q_t \right) \nonumber \\
&+ \textstyle\sum_{b\in\Acal}\nabla_{Y_b} V_a^\top f_b\big(t,\hat X_t,Y_t,u_t\big) \nonumber \\
&+ \tfrac{1}{2} \Tr\!\big( \nabla_{\hat{X}}^2 V_a \Pi_t \rho_t^\top R^{-1} \rho_t \Pi_t \big) \Big\} = 0.
\end{align}
Equation \eqref{eq:hjb} characterizes a single agent's best response; the collection $\{$\eqref{eq:hjb}$\}_{a\in\mathcal A}$, with $u_{-a}$ set to the equilibrium policies, constitutes the coupled Nash system, to be solved jointly as a fixed point. We state an assumption for the lower-level problem.

\begin{assumption}[Lower-level data]\label{ass:data}
The joint control decomposes into agent blocks, $u=(u_a)_{a\in\Acal}$, and each $c_a$ depends on $u$ only through its own block $u_a$. Each $c_a$ is $C^1$ and strictly convex in $u_a$; each $f_a$ is affine in $u$; and each admissible correspondence $U_a$ of \eqref{eq:caps} is nonempty, convex- and compact-valued, and Lipschitz in $\hat X$.
\end{assumption}

\noindent Two consequences should be stated. The admissible set factors as $\prod_a U_a$, so a unilateral deviation by one agent leaves the others' constraint sets untouched. And because the blocks are disjoint and each $c_a$ is strictly convex in $u_a$, the social cost $\sum_a c_a$ is strictly convex in the full control $u$.

We follow up with our result on Nash feedback below.

\begin{theorem}[Nash feedback strategies]\label{thm:nash_feedback}
Under Assumption~\ref{ass:data}, for each $a$ the pointwise minimization in \eqref{eq:hjb} admits a measurable selection $u_a^\star(t,\hat X,Y)$, and the profile $u^\star=(u_1^\star,\dots,u_N^\star)$ is a Nash feedback equilibrium. Componentwise:
\begin{enumerate}
\item[(i)] components of $u_a$ in which $c_a$ is strictly convex are given by a \emph{saturated} feedback, the unconstrained stationary point projected onto $U_a$, and are continuous in $(\hat X,Y)$;
\item[(ii)] components entering the bracket in \eqref{eq:hjb} \emph{linearly}, with no own running cost, are \emph{bang-bang}: they sit at an extreme point of $U_a$ determined by the sign of the associated marginal-value coefficient, with arbitrary tie-breaking where that coefficient vanishes. Across a surface on which a component of the coefficient changes sign, the selected extreme point changes; the feedback is therefore discontinuous there whenever the two extreme points are distinct, which holds unless $U_a$ is degenerate in that direction.
\end{enumerate}
\end{theorem}
\begin{proof}
Fix $a$ and consider the pointwise minimization over $u_a\in U_a(\hat X,\Nu)$ in \eqref{eq:hjb}. The admissible set is nonempty, convex and compact and the minimand is continuous, so a minimizer exists; measurability of a selection follows from Berge's maximum theorem together with the Kuratowski--Ryll-Nardzewski selection theorem, the correspondence $U_a$ being Lipschitz hence measurable (Assumption~\ref{ass:data}). This gives $u_a^\star$, and the profile $u^\star$ is a Nash feedback equilibrium because each component minimizes agent $a$'s own bracket given the others.

Split $u_a=(u_a^{\mathrm c},u_a^{\mathrm l})$ according to whether $c_a$ is strictly convex in that component, or the component enters the bracket linearly with no own running cost.

\emph{(i) Strictly convex components.} On these coordinates the minimand is the sum of the strictly convex $c_a$ and a term affine in $u_a^{\mathrm c}$, arising from $\sum_b\nabla_{Y_b}V_a^\top f_b$ with $f_b$ affine in $u$. The unconstrained first-order condition $\nabla_{u_a^{\mathrm c}}c_a = q$, where $q$ collects the affine coefficients, has a unique solution by strict convexity, and the constrained minimizer is its projection onto the convex compact set $U_a$. Since the projection onto a convex set is nonexpansive and $q$ depends continuously on $(\hat X,Y)$ through $\nabla V_a$, this component is a continuous, saturated feedback.

\emph{(ii) Linear components.} On these coordinates the minimand is affine, $\langle q_{\mathrm l},u_a^{\mathrm l}\rangle$ with $q_{\mathrm l}$ the associated marginal-value coefficient, so the minimum over the convex compact $U_a$ is attained at an extreme point determined by the sign pattern of $q_{\mathrm l}$: the feedback is bang-bang. Where a component of $q_{\mathrm l}$ vanishes every point of the corresponding face is optimal, so any measurable tie-breaking rule is admissible. If that component changes sign at such a point, the minimizing extreme point before and after are the two faces of $U_a$ opposite in that direction; they are distinct unless $U_a$ has empty extent there, so the selection jumps and the feedback is discontinuous across the surface.
\end{proof}

The bang-bang components of Theorem~\ref{thm:nash_feedback}(ii) render the optimal feedback discontinuous and the value functions generically non-$C^2$ across switching surfaces; the state-dependent admissible sets \eqref{eq:caps} compound this. Equations \eqref{eq:hjb} should therefore be interpreted in the viscosity sense~\cite{FlemingSoner}, and the optimality of the feedback of Theorem~\ref{thm:nash_feedback} requires a verification argument valid for non-smooth $V_a$. We treat \eqref{eq:hjb} as the formal characterization here; the full viscosity verification---existence, comparison, and a verification theorem valid for the non-smooth feedback---is carried out in Section~\ref{sec:viscosity}.

\section{Continuous Marginal-Contribution Transfers}\label{sec:transfers}
To deter strategic withholding of capability, the mechanism pays each agent its marginal contribution to the coupled system. Let $M_{a,t}$ be the aggregate cost reduction that agent $a$'s participation confers on the remaining agents under belief $\hat{X}_t$:

\begin{equation}\label{eq:Mdef}
M_{a,t} = \sum_{b \neq a} \big(V_b^{\text{NoCoupling}}(t, \hat{X}_t, Y_{b,t}) - V_b^{\text{Coupled}}(t, \hat{X}_t, Y_{b,t}) \big).
\end{equation}

The \emph{Groves} transfer to Agent $a$ is
\begin{equation}\label{eq:groves}
\pi_{a,t} = M_{a,t} - R_{a,t},
\end{equation}
where $R_{a,t}$ is any pivot term independent of $a$'s report (e.g.\ a Clarke pivot). The running net cost of $a$ is then $V_a^{\text{Coupled}} - M_{a,t} + R_{a,t} = -\big(\sum_{b}V_b^{\text{Coupled}} - \sum_{b\neq a}V_b^{\text{NoCoupling}}\big) + R_{a,t}$.

Existence of the lower-level Nash feedback equilibrium is \emph{not} assumed: it is established as a consequence of the Groves alignment, which renders the lower level a potential game whose potential is the social cost.

\begin{theorem}[Existence and uniqueness of the lower-level equilibrium]\label{thm:nash_exist}
Suppose Assumption~\ref{ass:data} holds and the agents face the Groves transfer \eqref{eq:groves}. Then, for every admissible disclosure policy $\rho\in\Ucal_L$:
\begin{enumerate}
\item[(i)] the social planner problem $\inf_{u\in\Ucal_O}\E\big[\int_0^T\ell\,\diff t\big]$ admits an optimal Markov feedback $u^\star(t,\hat X,Y)$;
\item[(ii)] the profile $u^\star$ is a Nash feedback equilibrium of the transfer-adjusted lower-level game \eqref{eq:cost};
\item[(iii)] the strictly convex components $P^\star$ are unique; the transfer components $T^\star$ are unique up to modification on the bang-bang switching set, which is Lebesgue-null by Proposition~\ref{prop:single_agent}(c).
\end{enumerate}
\end{theorem}
\begin{proof}
    See Appendix \ref{proof:theorem_nash_exist_proof}.
\end{proof}

\begin{theorem}[Efficiency and Incentive Compatibility]\label{thm:ic}
Consider the direct mechanism in which each agent reports its running cost $\hat c_a$ at time $0$, the coordinator computes the socially optimal Markov feedback for the \emph{reported} data (well-defined by Theorem~\ref{thm:nash_exist}), and transfers are given by \eqref{eq:groves}. Restrict deviations to Markov feedback strategies adapted to $\Gfilt_t$. Then, under the Groves transfer:
\begin{enumerate}
\item[(i)] \emph{(Reporting.)} Truthful reporting, $\hat c_a=c_a$, is a dominant strategy: for \emph{any} profile of rivals' reports, any misreport $C^D_{a}\neq C_{a}$ weakly raises $a$'s expected net cost.
\item[(ii)] \emph{(Obedience.)} Given truthful reports, following the efficient action profile $u^\star$ is a Nash best response for every agent: no unilateral Markov-feedback deviation from $u^\star_a$ improves $a$'s transfer-adjusted objective.
\item[(iii)] \emph{(Uniqueness up to null sets.)} The implemented equilibrium is unique up to Lebesgue-null modification: the strictly convex components are the unique minimizers of a strictly convex cost, and the bang-bang transfer components are determined uniquely off the switching set, which is Lebesgue-null by Proposition~\ref{prop:single_agent}(c). Any two equilibrium feedbacks therefore agree a.e.\ in $(t,\hat X,Y)$ and yield identical expected cost and transfers, since these are time integrals of the control along the closed loop and are unchanged by modification on a Lebesgue-null set; the equilibrium is payoff-unique.
\end{enumerate}
Consequently the efficient (welfare-maximizing) action is implemented: truthfully reporting and then executing $u^\star$ is incentive compatible, with dominance in the reporting dimension and Nash obedience in the control dimension, and the implemented equilibrium is payoff-unique.
\end{theorem}
\begin{proof}
    See Appendix \ref{proof:theorem_ic_proof}.
\end{proof}

\subsection{Budget balance and the Green--Laffont tradeoff}\label{sec:gl}
The Groves family \eqref{eq:groves} is efficient and dominant-strategy incentive compatible in reports but generically \emph{not} budget balanced: $\sum_a \pi_{a,t}\neq 0$, so the coordinator runs a surplus or deficit funded through a network insurance account $R_t := \sum_a R_{a,t}$. By the Green--Laffont impossibility~\cite{GreenLaffont1979}, no mechanism can simultaneously be efficient, dominant-strategy incentive compatible, and budget balanced. If exact budget balance is required, one must weaken solution concepts: the expected-externality (AGV) mechanism~\cite{AGV1979} achieves efficiency and \emph{Bayesian} incentive compatibility with exact budget balance under a common prior over $X_0$ and the jump/diffusion parameters. A recentered (Shapley-type) transfer $\pi_{a,t}^{\mathrm{BB}} = M_{a,t}-\tfrac1N\sum_c M_{c,t}$ restores $\sum_a \pi_{a,t}^{\mathrm{BB}}=0$ but reintroduces dependence of $a$'s payoff on its own report through the $\{M_c\}_{c\neq a}$ terms, and hence yields only approximate, not exact, incentive compatibility. We adopt the Groves/AGV pairing to make this tradeoff explicit rather than to claim all three properties at once.

\section{The Stackelberg Master Problem}\label{sec:master}
We can now state the control problem compactly. With $\ell$ the social running cost \eqref{eq:ell} and $u^\star(\rho)$ the lower-level Nash response, the coordinator (Stackelberg leader) solves the bilevel problem
\begin{equation}\label{eq:problemP}\tag{P}
\inf_{\rho\in\mathcal{U}_L}\ \E\!\left[\int_0^T \ell\big(t,\hat X_t,\Pi_t,Y_t,u^\star_t(\rho),\rho_t\big)\,\diff t\right]
\end{equation}
subject to the belief-mean filter~\eqref{eq:meanfilter}, the deterministic Riccati flow~\eqref{eq:riccati}, the local-state dynamics~\eqref{eq:Ydyn} and the admissibility constraint~\eqref{eq:caps}, and the equilibrium constraint that $u^\star(\rho)$ be the Nash feedback of Theorem~\ref{thm:nash_feedback} under the Groves transfer~\eqref{eq:groves}; here $\mathcal{U}_L$ denotes the admissible disclosure gains and the expectation is under the reference measure $\Pprob$. Problem~\eqref{eq:problemP} is a Stackelberg bilevel problem; Sections~\ref{sec:bilevel}--\ref{sec:viscosity} show that incentive alignment collapses it to a single stochastic control problem whose value is the unique viscosity solution of the Hamilton--Jacobi--Bellman equation derived below.

Substituting the lower-level Nash profile, problem~\eqref{eq:problemP} is realized as
\begin{align}\label{eq:master}
\min_{\rho_t} \ \E \bigg[ \int_0^T \Big( \sum_{a \in \Acal} c_a\big(t,\hat X_t,Y_t,u^\ast_t\big) + \Tr(\rho_t \Lambda \rho_t^\top) \Big) \diff t \bigg],
\end{align}
subject to the filter dynamics \eqref{eq:meanfilter}--\eqref{eq:riccati} and the lower-level Nash profile $u^\ast$ of Theorem~\ref{thm:nash_feedback}. Here $\Lambda\succ0$ prices disclosure.

Write $w:=(\hat X,Y)$ for the stochastic part of the state, with $\Pi$ a coefficient carried alongside, and define the \emph{Stackelberg master value}

\begin{equation}\label{eq:value}
S(t,w) \;=\;  \inf_{(\rho,u)\in\Ucal_L\times\Ucal_O} \E\!\Big[\int_t^T \ell(s,w_s,u_s,\rho_s)\,\diff s \,\Big|\, w_t=w\Big],
\end{equation}

the infimum being over admissible disclosure policies and lower-level controls. The following characterizes $S$.

\begin{theorem}[The Master Hamilton--Jacobi--Bellman Equation]\label{thm:master_robust}
Let $\mathcal D=[0,T]\times\R^n\times\Sym^n_+\times\R^N$ and let $S$ be the Stackelberg master value of \eqref{eq:value}. Then $S$ satisfies the dynamic-programming principle, and if in addition $S\in C^{1,2}(\mathcal D)$ with polynomial growth, $S$ solves the partial integro-differential Hamilton--Jacobi--Bellman equation
\begin{align}\label{eq:isaacs}
\frac{\partial S}{\partial t}
&+\min_{\rho\in\Ucal_L}\Big\{ \textstyle\sum_{a\in\Acal}\mathcal{C}_a^{\text{macro}}
+ \nabla_{\hat{X}} S^\top\big(A_t \hat{X} + B \hat Q_t\big) \nonumber\\
& + \tfrac{1}{2}\Tr\!\big(\nabla_{\hat{X}}^2 S\,\Pi\rho^\top R^{-1}\rho\,\Pi\big)
+ \Tr\!\big(\partial_\Pi S\;\mathcal R(\Pi,\rho)\big) \nonumber\\
& + \textstyle\sum_{a\in\Acal}\nabla_{Y_a}S^\top \dot Y_{a}\Big\}
\;+\;\lambda\,\Ical S \;=\;0,
\end{align}
on $[0,T)\times\R^n\times\Sym^n_+\times\R^N$, with terminal condition $S(T,\cdot)=0$. Here
\[
\mathcal R(\Pi,\rho)=A_t\Pi+\Pi A_t^\top+\Sigma_t\Sigma_t^\top-\Pi\rho^\top R^{-1}\rho\Pi
\]
is the inter-epoch Riccati drift of \eqref{eq:riccati} \emph{without} the predictive rate $\lambda\Sigma_J$, which appears instead in the nonlocal term
\begin{equation}\label{eq:jumpop}
\Ical S(t,\hat X,\Pi,Y):=S\big(t,\hat X,\Pi+\Sigma_J,Y\big)-S\big(t,\hat X,\Pi,Y\big),
\end{equation}
and $\mathcal{C}_a^{\text{macro}} = c_a\big(t,\hat X,Y,u^\ast\big) + \Tr(\rho\Lambda\rho^\top)/N$ collects the running cost, with $\dot Y_{a}$ given by \eqref{eq:Ydyn} at the lower-level equilibrium.
\end{theorem}

\begin{proof}
    See Appendix \ref{proof:theorem_master_proof}.
\end{proof}

\begin{remark}[The nonlocal term is a difference, not an integral]\label{rem:jumpop}
Because the marks are unobserved and centred (Assumption~\ref{ass:observed_jumps} and the discussion in Section~\ref{sec:signal}), an epoch leaves the belief mean fixed and inflates the covariance by the \emph{deterministic} amount $\Sigma_J$. The jump operator \eqref{eq:jumpop} is therefore a bounded \emph{difference} agent acting only on the $\Pi$-argument, not an integral against the mark law. This is what keeps \eqref{eq:isaacs} within reach of the standard theory: $\Ical$ is nonlocal but order-zero, and at an interior maximum of $\underline S-\overline S$ one has $\Ical\underline S-\Ical\overline S\le0$, which is exactly the monotonicity a comparison principle for integro-differential equations requires~\cite{BarlesImbert2008}. Section~\ref{sec:viscosity} uses only this property.
\end{remark}

The master \eqref{eq:isaacs} is coupled to the followers through $\{P^\ast_{a,g,t},T^\ast_{ba,t}\}$, which depend on $\{\nabla_Y V_a\}$ via Theorem~\ref{thm:nash_feedback}. A fully rigorous treatment therefore solves the bilevel system $(\{V_a\}, S)$ jointly; the leader's first-order conditions in $\rho_t$ must account for the sensitivity of the lower-level equilibrium to $\rho_t$ (an implicit-function / adjoint argument). We characterize the inner problem here; the joint bilevel verification---including the exact treatment of this sensitivity via an envelope argument---is carried out in Section~\ref{sec:bilevel}.

\subsection{The two coordination levers as nested limits}\label{rem:infocontrol}
The master equation \eqref{eq:isaacs} contains, as limits obtained by disabling one instrument at a time, the two levers by which a coordinator steers strategic agents without commanding them~\cite{BasarHayakawaIshiiZhu2026}: information control and incentive design. Neither lever is new as a subject. What the master equation supplies is a single dynamic-programming object from which both are recovered, and with it a control formulation of the first (Section~\ref{sec:related_gap}).

\emph{(i) Information control.} Suppose the followers require no incentive correction, either because a single agent is coordinated or because their objectives already coincide with the social cost, so that the transfer layer is inactive. The leader's only instrument is then the disclosure gain $\rho_t$, and \eqref{eq:isaacs} reduces to
\begin{align*}
\partial_t S + \nabla_{\hat X}S^\top(A_t\hat X_t + B\hat Q_t) + \ell_0 \; + \min_{\rho_t}\Big\{ \Tr(\rho_t\Lambda\rho_t^\top) + \tfrac12\Tr\big(\nabla^2_{\hat X}S\,\Lambda_{\rho}\big) + \Tr\big(\partial_\Pi S\,\dot\Pi_\rho\big) \Big\} = 0,
\end{align*}
with $\Lambda_\rho=\Pi_t\rho_t^\top R^{-1}\rho_t\Pi_t$ and $\ell_0$ the control-independent running cost. This is a stochastic optimal-control problem in which the \emph{state is a belief}: the mean $\hat X_t$ diffuses with a control-dependent volatility, and the covariance $\Pi_t$ follows the deterministic Riccati flow \eqref{eq:riccati}, whose gain term $-\Pi_t\rho_t^\top R^{-1}\rho_t\Pi_t$ is the only channel through which $\rho_t$ acts. Disclosure therefore \emph{steers uncertainty} along a trajectory in the cone $\Sym^n_+$, at a quadratic input cost $\Tr(\rho_t\Lambda\rho_t^\top)$, and the leader trades the marginal cost of transparency $2\Lambda\rho_t$ against the marginal value of uncertainty reduction transmitted through $\partial_\Pi S$ and $\nabla^2_{\hat X}S$. The structural results of Section~\ref{sec:viscosity} specialize without change: the Hamiltonian remains proper and degenerate elliptic (Section~\ref{sec:viscosity}), the Riccati bounds of Lemma~\ref{lem:riccati} still confine $\Pi$ to the interior of the cone, and the value remains the unique viscosity solution. Because the bang-bang transfer term is what destroys smoothness, the reduced problem is in fact better behaved: with $u$ absent the Hamiltonian is smooth in $p$ and the switching set treated in Section~\ref{sec:viscosity} is empty. It must be stressed that this reduction requires the followers to need no incentive correction. Merely \emph{deleting} the transfers while the followers remain strategic and misaligned does not produce \eqref{eq:isaacs} with $u$ frozen at the social optimum; it produces the equilibrium-constrained problem \eqref{eq:ICDG} of Section~\ref{sec:icdg}, in which the lower-level Nash response is not the social optimizer and the envelope identity established in Section~\ref{sec:bilevel} fails.

\emph{(ii) Incentive design.} Freezing $\rho_t\equiv\rho$ removes the information lever and leaves a dynamic marginal-contribution mechanism inside the differential game of Section~\ref{sec:game}: the belief filter runs at a fixed precision, $\Pi_t$ is an exogenous coefficient, and the efficiency collapse established in Section~\ref{sec:bilevel} still applies, so the Groves transfer implements the efficient action profile under a fixed information structure. This limit is the continuous-time counterpart of the dynamic pivot mechanism~\cite{BergemannValimaki2010}, which pays each agent its marginal contribution under an exogenous information structure; it is recovered here, not introduced.

The joint problem is not the superposition of these limits, and this is the substantive point. Section~\ref{sec:special} shows that under a one-factor common-shock structure the two instruments are complements: the marginal value of disclosure depends on whether the agents are coupled, and conversely. A coordinator who optimizes the disclosure channel taking the mechanism as given, or the mechanism taking the information structure as given, therefore does not reach the joint optimum in this case.

\section{Bilevel Verification}\label{sec:bilevel}
This section resolves the bilevel-coupling gap: it gives conditions under which the Stackelberg leader's transparency policy and the followers' Nash equilibrium are \emph{jointly} verified by a single dynamic-programming object, and it makes precise the sense in which the leader's first-order condition in $\rho_t$ correctly accounts for the sensitivity of the lower-level equilibrium. Throughout, $\Sym^n_+$ denotes symmetric positive semidefinite matrices.

\subsection{Reduced state and admissible controls}
Let $\beta_t$ be the $\Gfilt_t$-innovation Brownian motion of the belief filter, so that the belief mean obeys
\begin{equation}\label{eq:beliefSDE}
\begin{gathered}
\diff \hat X_t = \big(A_t \hat X_t + B\hat Q_t\big)\,\diff t + G_\rho(t)\,\diff\beta_t,\\
G_\rho(t) := \Pi_t\,\rho_t^\top R^{-1/2},
\end{gathered}
\end{equation}
and $G_\rho G_\rho^\top = \Pi_t\rho_t^\top R^{-1}\rho_t\Pi_t =: \Lambda_\rho(t)$, while the error covariance $\Pi_t$ solves the matrix Riccati equation \eqref{eq:riccati}. The local states evolve without diffusion,
\begin{align}
\diff Y_{a,t} &= f_a\big(t,\hat X_t, Y_t, u_t\big)\,\diff t,\label{eq:Yred}
\end{align}
where $u=(u_a)_{a\in\Acal}$ is the stacked agent control and $w_t=(\hat X_t,Y_t)$ is the stochastic part of the state, as in Section~\ref{sec:master}.

\begin{definition}[Admissible controls]\label{def:admissible}
$\Ucal_L$ is the set of measurable transparency policies $\rho:[0,T]\to\R^{m\times n}$ for which \eqref{eq:riccati} has a unique solution $\Pi\in C([0,T];\Sym^n_+)$ and $\int_0^T\!\Tr(\rho_t\Lambda\rho_t^\top)\diff t<\infty$. $\Ucal_O$ is the set of progressively measurable feedback policies $u(t,w)$ valued in the (compact, state-dependent) constraint set of \eqref{eq:caps} for which \eqref{eq:beliefSDE}--\eqref{eq:Yred} admits a unique strong solution. 
\end{definition}

\subsection{The efficiency collapse}

Let the social running cost be
\begin{equation}\label{eq:socialcost}
\ell(t,w,\Pi,u,\rho) \;=\; \sum_{a\in\Acal} c_a\big(t,\hat X,Y,u\big) + \Tr(\rho\Lambda\rho^\top),
\end{equation}

in which the inter-agent transfers do not appear: under any budget rule they are internal redistributions that cancel in the aggregate.

\begin{assumption}[Common reference model]\label{ass:common}
The agents and the leader evaluate expectations under the same reference measure $\Pprob$: no agent holds a private model of the latent dynamics or of the signal channel beyond the private information already carried by its report.
\end{assumption}

\begin{theorem}[Efficiency collapse]\label{thm:collapse}
Under Assumptions~\ref{ass:data} and~\ref{ass:common} and the Groves transfer of Theorem~\ref{thm:ic}, for each $\rho\in\Ucal_L$ the lower-level Nash equilibrium $u^\star(\rho)$ coincides with the unique minimizer of the social cost
$\E\!\big[\int_0^T \ell\,\diff t\big]$ over $u\in\Ucal_O$. Consequently the Stackelberg value satisfies
\begin{align}\label{eq:collapse}
\inf_{\rho\in\Ucal_L}\ \E\!\Big[\int_0^T \ell\big(t,w_t,\Pi_t,u^\star(\rho)_t,\rho_t\big)\diff t\Big] \;=\; \inf_{(\rho,u)\in\Ucal_L\times\Ucal_O}\ \E\!\Big[\int_0^T \ell\,\diff t\Big],
\end{align}
the value of a single joint optimal-control problem whose dynamic-programming equation is the master HJB equation \eqref{eq:isaacs} with the pointwise optimization taken jointly over $(\rho,u)$.
\end{theorem}
\begin{proof}
By the Groves construction each agent's private objective differs from the social objective by terms independent of that agent's own report and control (the pivot $R_{a,t}$ and the rivals' no-coupling baselines). Hence minimizing the private objective is equivalent to minimizing the social objective \eqref{eq:socialcost}, and truthful reporting is a dominant strategy and efficient operation the induced best response (Theorem~\ref{thm:ic}). By Assumption~\ref{ass:data} the social problem is strictly convex in $u$ with compact feasible set, so its minimizer $u^\star(\rho)$ is unique; being each agent's best response, it is the unique Nash equilibrium. Substituting $u^\star(\rho)$ into the leader's objective gives the left side of \eqref{eq:collapse} equal to $\inf_\rho\inf_u \E[\int\ell]$, and combining the two infima yields the joint infimum on the right. The dynamic-programming principle for the joint control problem produces \eqref{eq:isaacs} with the joint pointwise $\min_{(\rho,u)}$ Hamiltonian.
\end{proof}

The economic content is that incentive alignment is not merely a fairness device: it is what makes the leader's problem a standard control problem rather than a genuine bilevel program. With efficient followers the coordinator effectively controls the whole network directly, subject only to the agents carrying out the socially optimal action.

\subsection{The verification theorem}

Let $\mathcal D := [0,T]\times\R^n\times\Sym^n_+\times\R^N$ and, for $S\in C^{1,2}(\mathcal D)$, define the joint Hamiltonian
\begin{align}\label{eq:jointH}
\Hcal\big(t,w,\Pi,\nabla S,\nabla^2 S;\,\rho,u\big) &= \ell(t,w,\Pi,u,\rho)
+ \nabla_{\hat X}S^\top\!\big(A_t\hat X + B\hat Q_t\big) \nonumber\\
& + \sum_{a}\nabla_{Y_a}S\, f_a(\hat X,u)
+ \tfrac12\Tr\!\big(\nabla^2_{\hat X}S\,\Lambda_\rho\big)  + \Tr\!\big(\partial_\Pi S\,\dot\Pi_\rho\big),
\end{align}
where $\dot\Pi_\rho=\mathcal R(\Pi,\rho)$ is the inter-epoch Riccati drift of Theorem~\ref{thm:master_robust}. The master HJB equation is $\partial_t S + \min_{(\rho,u)}\Hcal + \lambda\Ical S = 0$ with $S(T,\cdot)=0$, the nonlocal term $\Ical$ of \eqref{eq:jumpop} being independent of $(\rho,u)$ and therefore outside the minimization.

We first verify the bilevel structure under a classical smoothness hypothesis on $S$, isolating the efficiency-collapse and envelope arguments; Section~\ref{sec:viscosity} removes this hypothesis and treats the generic non-smooth value function in the viscosity sense.

\begin{assumption}[Classical regularity]\label{ass:regular}
There exists $S\in C^{1,2}(\mathcal D)$ with at most polynomial growth solving \eqref{eq:isaacs}, and the pointwise minimization $\min_{(\rho,u)}\Hcal$ admits a measurable minimizer $\big(\rho^\star,u^\star\big)(t,w,\Pi)$ with $\rho^\star\in\Ucal_L$, $u^\star\in\Ucal_O$, and the closed-loop system \eqref{eq:beliefSDE}--\eqref{eq:Yred} under $(\rho^\star,u^\star)$ has a unique strong solution.
\end{assumption}

\begin{theorem}[Bilevel verification]\label{thm:verify}
Under Assumptions~\ref{ass:data}--\ref{ass:regular} and Theorem~\ref{thm:collapse}:
\begin{enumerate}
\item[(i)] $S(0,w_0,\Pi_0)$ equals the optimal Stackelberg value \eqref{eq:collapse};
\item[(ii)] $\rho^\star$ is an optimal transparency policy for the coordinator;
\item[(iii)] along the optimal trajectory the agent controls $u^\star$ form a Nash equilibrium of the lower-level game.
\end{enumerate}
Moreover the leader's stationarity condition $\partial_\rho\Hcal=0$ is \emph{exact}: at the joint optimum the envelope identity $\partial_u\Hcal=0$ implies that the implicit response $\partial u^\star/\partial\rho$ contributes nothing to the leader's first-order condition.
\end{theorem}
\begin{proof}
(i)--(ii). Fix any admissible $(\rho,u)$. Applying It\^o's formula to $S(t,w_t,\Pi_t)$ between $0$ and $T$, taking $\E$, and using \eqref{eq:beliefSDE}--\eqref{eq:Yred} and \eqref{eq:riccati},
\[
\E\big[S(T,w_T,\Pi_T)\big] - S(0,w_0,\Pi_0)
= \E\!\int_0^T\!\Big(\partial_t S + \Hcal - \ell\Big)\diff t,
\]
where the martingale term vanishes by the integrability in $\Ucal_O$. Since $S(T,\cdot)=0$ and, by \eqref{eq:isaacs}, $\partial_t S + \Hcal\ge \partial_t S + \min_{(\rho,u)}\Hcal = 0$ with equality at $(\rho^\star,u^\star)$, rearranging gives
$S(0,w_0,\Pi_0)\le \E\!\int_0^T\ell\,\diff t$ for every admissible $(\rho,u)$, and equality is attained by the minimizing policies. Taking $\inf_{(\rho,u)}$ yields (i); the attaining policies give (ii); this is the standard stochastic verification argument~\cite{YongZhou,FlemingSoner}.

(iii). By Theorem~\ref{thm:collapse} the inner minimizer $u^\star$ of the joint problem is precisely the social optimizer, which under the Groves alignment is each agent's best response; hence $u^\star$ is a Nash equilibrium of the lower-level game along the optimal path.

Envelope statement. Write the leader's reduced value $\Vcal(t,w,\Pi;\rho)= \min_u\{\cdots\}$, the inner value of \eqref{eq:jointH} at fixed $\rho$. The inner minimization is attained at $u^\star(\rho)$ with $\partial_u\Hcal|_{u^\star}=0$ (first-order condition for the convex inner problem, Assumption~\ref{ass:data}). By Danskin's envelope theorem~\cite{Danskin} the total derivative of $\Vcal$ in $\rho$ equals the partial derivative at $u^\star$, $\,\mathrm d\Vcal/\mathrm d\rho = \partial_\rho\Hcal\big|_{u^\star(\rho)}$, so the implicit response $\partial u^\star/\partial\rho$ does not enter. The leader's stationarity is therefore $\partial_\rho\Hcal=0$, balancing the marginal disclosure cost $2\Lambda\rho$ against the marginal value of uncertainty reduction transmitted through $\tfrac12\Tr(\nabla^2_{\hat X}S\,\partial_\rho\Lambda_\rho)+\Tr(\partial_\Pi S\,\partial_\rho\dot\Pi_\rho)$.
\end{proof}

\begin{remark}[Regularity and the link to the viscosity treatment]\label{rem:reg}
Assumption~\ref{ass:regular} requires a classical $C^{1,2}$ solution of \eqref{eq:isaacs}. The saturated and bang-bang feedback of Theorem~\ref{thm:nash_feedback} renders the agents' value functions---and hence $S$ through the coupling---non-smooth across switching surfaces, so \eqref{eq:isaacs} generally admits only a viscosity solution. The verification argument extends to that setting by replacing It\^o's formula with the sub/supersolution inequalities for viscosity solutions and invoking a comparison principle for the HJB operator; this is carried out in Section~\ref{sec:viscosity} and removes Assumption~\ref{ass:regular}.
\end{remark}

\section{Viscosity Solutions and Verification under Non-Smooth Feedback}\label{sec:viscosity}
This section removes the classical-regularity hypothesis (Assumption~\ref{ass:regular} of Section~\ref{sec:bilevel}). Throughout, we treat a standard stochastic control problem parameterized by an arbitrary admissible disclosure policy and control profile; \emph{all existence, comparison, and verification results in this section are proved for arbitrary admissible controls.} The saturated feedback on the strictly convex block and, especially, the bang-bang feedback on the linear block make the value functions non-$C^2$, so the master HJB equation \eqref{eq:isaacs} is interpreted in the viscosity sense. We (i) record the structure of the optimized Hamiltonian, (ii) show the value function is the unique viscosity solution via a comparison principle, (iii) prove a verification theorem valid for the non-smooth value function, and (iv) establish that the discontinuous feedback is defined almost everywhere and generates a unique closed-loop trajectory. We work on the joint control problem produced by the efficiency collapse. Two structural facts frame the analysis. First, the error covariance $\Pi$ is \emph{not} a state variable in the viscosity-theoretic sense: given the leader's policy $\rho(\cdot)$ and conditionally on the observed epoch path, $\Pi$ is the fixed, deterministic solution of the Riccati flow \eqref{eq:riccati} on each inter-jump interval (Lemma~\ref{lem:reduction}), so all test functions, sub/superdifferentials, and the doubling of variables below act only on the genuinely stochastic coordinates $(\hat X,Y)\in\R^{n+N}$, with $\Pi_t$ entering the Hamiltonian as a known, bounded, Lipschitz time-dependent coefficient. In particular no boundary condition on the cone $\Sym^n_+$ is required. The analysis is organized interval by interval, in the manner of piecewise-deterministic control: on each $[\tau_k,\tau_{k+1})$ the coefficient $\Pi$ is deterministic and the comparison and verification arguments below apply verbatim; at each observed reset the value function is glued by the dynamic-programming matching condition $S(\tau_k^-,\hat X,Y)\,{=}\,S(\tau_k,\hat X,Y)$ evaluated with $\Pi_{\tau_k}=\Pi_{\tau_k^-}+\Sigma_J$, since the observed epoch changes no cost instantaneously and the jump times are totally inaccessible. Backward induction over the at most $K$ intervals (Assumption~\ref{ass:riccati}) then delivers the results on all of $[0,T]$. Second, this reduction is legitimate because the Riccati trajectory never touches the cone boundary.

With these facts in hand, the state for the viscosity analysis is again $w=(\hat X,Y)\in\R^{n+N}$, with $\Pi_t$ a coefficient; we retain the notation $w=(\hat X,Y,\Pi)\in\mathcal O$ only where convenient, always with $\Pi$ frozen at its Riccati value.

\begin{remark}[Invariance under null-set modification]\label{rem:nullset}
The expected running cost and transfers are time integrals against the law of the closed-loop process, so they are unchanged if the feedback is modified on a set that the process occupies for zero expected time. In particular, modifying a feedback on a Lebesgue-null set of $(t,\hat X,Y)$ leaves the expected cost and transfers unchanged. Throughout, optimal and equilibrium feedbacks are understood up to such null modification, and the switching set of the bang-bang controls, Lebesgue-null by Proposition~\ref{prop:single_agent}(c), carries no payoff consequence. We claim payoff-equivalence only; pathwise identity of trajectories is neither asserted nor used.
\end{remark}

\subsection{Structure of the optimized Hamiltonian}\label{sub_opt_ham}

For $S\in C^{1,2}$ equation \eqref{eq:isaacs} reads $\partial_t S + \bar H(t,w,\nabla S,\nabla^2_{\hat X}S) + \lambda\Ical S=0$, $S(T,\cdot)=0$, with
\begin{align}\label{eq:Hbar}
\bar H(t,w,p,M) = \min_{(\rho,u)} \big\{ \ell + p_{\hat X}^\top(A_t\hat X + B\hat Q_t) + \textstyle\sum_a p_{Y_a} f_a(\hat X,u) + \Tr(p_\Pi\dot\Pi_\rho) + \tfrac12\Tr(M\Lambda_\rho) \big\},
\end{align}
where $p=(p_{\hat X},p_Y,p_\Pi)$, $\Lambda_\rho=\Pi\rho^\top R^{-1}\rho\Pi$, and $\dot\Pi_\rho$ is the Riccati right-hand side. Every term is now affine in $p_{\hat X}$ and $p_\Pi$; the only non-smoothness is in the $u$-minimization, which splits:

\begin{itemize}
\item \textbf{Strictly convex components.} Write $u_a=(u_a^{\mathrm c},u_a^{\mathrm l})$, where $u_a^{\mathrm c}$ collects the components in which $c_a$ is strictly convex. Minimizing $c_a(u_a^{\mathrm c})-p_{Y_a}^\top B_a^{\mathrm c}u_a^{\mathrm c}$ over the corresponding compact convex section of $U_a$ is a constrained Legendre transform of a strictly convex $C^1$ function, whose optimal value is $C^{1,1}$ in $p_{Y_a}$: \emph{Lipschitz gradient}, not $C^2$.
\item \textbf{Linear components.} The remaining components $u_a^{\mathrm l}$ carry no own running cost and enter the bracket affinely, through $p_Y^\top B^{\mathrm l}u^{\mathrm l}$. Minimizing an affine function over a compact convex set returns, by definition of the support function, minus that support function, $-\,\sigma_{U^{\mathrm l}}(-B^{\mathrm l\top}p_Y)$, which is \emph{convex, positively homogeneous, and globally Lipschitz} in $p_Y$, with kinks confined to the set where the minimizing vertex changes. When $U^{\mathrm l}$ is a product of intervals this is a finite max of affine functions of $p_Y$, and the kink set is the union of the hyperplanes on which the corresponding coefficient of $p_Y$ vanishes.
\end{itemize}

\begin{lemma}[Admissible structure of $\bar H$]\label{lem:struct}
Under the standing assumptions (Lipschitz, bounded $A_t,B,\hat Q_t$; $f_a$ affine in $u$ and Lipschitz in $(\hat X,Y)$; $U_a$ compact, convex, Lipschitz in $\hat X$; each $c_a$ $C^1$ and strictly convex in $u_a$), $\bar H$ in \eqref{eq:Hbar} is finite and continuous on $[0,T]\times\mathcal O\times\R^{n+N+\dim\Pi}\times\Sym^n$, and satisfies:
\begin{enumerate}
\item[(H1)] \emph{Properness / degenerate ellipticity:} $\bar H(t,w,p,M)\le \bar H(t,w,p,M')$ whenever $M\preceq M'$ (the only second-order term, $\tfrac12\Tr(M\Lambda_\rho)$, is monotone in $M$ since $\Lambda_\rho\succeq0$, and the pointwise minimum preserves monotonicity);
\item[(H2)] \emph{Global Lipschitz in the gradient:} $|\bar H(t,w,p,M)-\bar H(t,w,q,M)|\le C(1+|w|)\,|p-q|$, uniformly on $\Sym^n$-bounded sets of $M$ ($\bar H$ is affine in $p_{\hat X},p_\Pi$ and piecewise-linear/$C^{1,1}$ in $p_Y$, so no quadratic-in-$p$ term is present);
\item[(H3)] \emph{State modulus of continuity:} there is a modulus $\omega$ with
$\bar H(t,w',p,M')-\bar H(t,w,p,M)\le \omega\big(|w-w'|(1+|p|)+\eta\big)$
whenever $(M,M')$ are the matrices furnished by the Crandall--Ishii lemma at scale $\eta$;
\item[(H4)] \emph{Linear growth:} $|\bar H(t,w,0,0)|\le C(1+|w|)$.
\end{enumerate}
\end{lemma}
\begin{proof}
Continuity and finiteness follow from compactness of the control sets and the maximum theorem. (H1) is immediate from the displayed second-order term and monotonicity of the pointwise minimum. (H2): the drift terms $p_{\hat X}^\top(A_t\hat X+B\hat Q_t)$ and $\Tr(p_\Pi\dot\Pi_\rho)$ are affine in $p$ with $C(1+|w|)$-Lipschitz coefficients; the box-constrained convex term is $C^{1,1}$ in $p_Y$ with gradient bounded by the radius of $U_a$; the linear block contributes a term Lipschitz in $p_Y$ with constant $\sup_{u\in U}\|u\|$; a minimum of $L$-Lipschitz families is $L$-Lipschitz. There is no quadratic-in-$p$ term, so the bound is \emph{global}, not merely local. (H3) follows from Lipschitz dependence of $\ell$, $f_a$, $A_t\hat X$, $\Lambda_\rho$, $\dot\Pi_\rho$ on $w$ and the standard estimate for the Ishii matrices. (H4) is the bound at $p=0,M=0$ from the bounded running cost on compact control sets.
\end{proof}

Lemma~\ref{lem:struct} places the equation in a standard class: the bang-bang non-smoothness lands entirely in the $p$-dependence as a \emph{Lipschitz} kink, never in the ellipticity, so the equation is a proper, globally Lipschitz, degenerate-parabolic HJB equation of exactly the class covered by the Crandall--Ishii--Lions theory~\cite{CrandallIshiiLions}. The nonlocal term $\lambda\Ical$ is order-zero, bounded, and monotone: $\Ical$ evaluates its argument at the shifted covariance $\Pi+\Sigma_J$, so at a maximum of a difference of functions the nonlocal terms contribute with a favourable sign (Section~\ref{sec:master}); it is handled by the integro-differential extension of that theory~\cite{BarlesImbert2008} and is carried through the arguments below without further comment.

\subsection{Viscosity solution, existence, and comparison}

\begin{definition}[Viscosity solution; sign convention]\label{def:visc}
Write the terminal-value problem \eqref{eq:isaacs} in the Crandall--Ishii--Lions form. Define
\begin{equation*}
F\big(t,w,\partial_tS,\nabla S,\nabla^2_{\hat X}S\big)\;:=-\,\partial_t S\;-\;\bar H\big(t,w,\nabla S,\nabla^2_{\hat X}S\big),
\end{equation*}
and impose
\begin{equation*}
F=0\ \ \text{on }[0,T)\times\R^{n+N},\qquad S(T,\cdot)=0 .
\end{equation*}
The equation holds on the open time interval; the terminal condition is imposed at $t=T$, which $[0,T)$ excludes.
By (H1) of Lemma~\ref{lem:struct}, $\bar H$ is nondecreasing in $M$, so $F$ is nonincreasing in $M$: $F$ is \emph{proper and degenerate elliptic} in the sense of~\cite{CrandallIshiiLions}, and the standard conventions apply verbatim. An upper semicontinuous $S$ of polynomial growth is a \emph{viscosity subsolution} if for every $\varphi\in C^{1,2}$ and every local maximum $(t_0,w_0)$ of $S-\varphi$ with $t_0<T$,
\begin{gather*}
F\big(t_0,w_0,\partial_t\varphi,\nabla\varphi,\nabla^2_{\hat X}\varphi\big)\le0,\qquad\text{equivalently}\\
\partial_t\varphi(t_0,w_0)+\bar H\big(t_0,w_0,\nabla\varphi,\nabla^2_{\hat X}\varphi\big)\ge0,
\end{gather*}
and $S(T,\cdot)\le0$. A lower semicontinuous $S$ of polynomial growth is a \emph{viscosity supersolution} if at every local minimum $(t_0,w_0)$ of $S-\varphi$ with $t_0<T$,
\begin{gather*}
F\big(t_0,w_0,\partial_t\varphi,\nabla\varphi,\nabla^2_{\hat X}\varphi\big)\ge0,\qquad\text{equivalently}\\
\partial_t\varphi(t_0,w_0)+\bar H\big(t_0,w_0,\nabla\varphi,\nabla^2_{\hat X}\varphi\big)\le0,
\end{gather*}
and $S(T,\cdot)\ge0$. A \emph{viscosity solution} is both. The signs are those of the backward (terminal-value) parabolic form: at a maximum of $S-\varphi$ the dynamic-programming inequality yields precisely the subsolution inequality for the value function (Fleming--Soner~\cite{FlemingSoner}, Ch.~V).
\end{definition}

\begin{theorem}[Master HJB equation in the viscosity sense]\label{thm:existence}
Under the standing assumptions the value function \eqref{eq:value} is continuous, of at most quadratic growth, and is a viscosity solution of \eqref{eq:isaacs} with terminal data $S(T,\cdot)=0$.
\end{theorem}
\begin{proof}
Continuity and the growth bound follow from boundedness of the controls, Lipschitz coefficients, and Gronwall estimates on \eqref{eq:value}. The dynamic-programming principle,
$S(t,w)=\inf_{(\rho,u)}\E[\int_t^{t+h}\ell\,\diff s + S(t+h,w_{t+h})]$,
holds by the standard measurable-selection argument. Testing the DPP against a smooth $\varphi$ touching $S$ from above (resp.\ below) and dividing by $h\downarrow0$ gives the sub- (resp.\ super-) solution inequality with Hamiltonian \eqref{eq:Hbar}; see Fleming--Soner~\cite{FlemingSoner}, Ch.~V, and Bardi--Capuzzo-Dolcetta~\cite{BardiCapuzzo} for the first-order components.
\end{proof}

Theorem \ref{thm:existence} is the counterpart of Theorem~\ref{thm:master_robust} with its $C^{1,2}$ conditional removed: the saturated and bang-bang lower-level feedback renders $S$ non-$C^{1,2}$ across switching surfaces, and the equation is then satisfied in the viscosity sense.

\begin{theorem}[Comparison and uniqueness]\label{thm:comparison}
Let $\underline S$ be an upper semicontinuous viscosity subsolution and $\overline S$ a lower semicontinuous viscosity supersolution of \eqref{eq:isaacs}, both of polynomial growth, with $\underline S(T,\cdot)\le\overline S(T,\cdot)$. Then $\underline S\le\overline S$ on $[0,T]\times\mathcal O$. Consequently \eqref{eq:isaacs} has at most one viscosity solution of polynomial growth, namely the value function \eqref{eq:value}.
\end{theorem}
\begin{proof}
Assume for contradiction $\sup(\underline S-\overline S)>0$. Use the change $\tilde S=\exp(\kappa t)S$ to obtain strict monotonicity in the zeroth-order term, the polynomial-growth penalization $-\varepsilon \exp(\lambda(T-t))(1+|w|^2)$ to localize on the unbounded $\mathcal O$, and the variable-doubling penalty $\tfrac{|w-w'|^2}{2\eta}$. The Crandall--Ishii lemma supplies $M\preceq M'$ at the doubling maximum, and property (H1) (degenerate ellipticity) controls the second-order difference, (H2)--(H3) control the gradient and state terms, and the $\exp(\kappa t)$ factor controls the zeroth order. At the penalized maximum the nonlocal terms satisfy $\Ical\underline S-\Ical\overline S\le0$, since $\Ical$ is the difference operator of Section~\ref{sec:master} evaluated at the common shifted argument $\Pi+\Sigma_J$, so $\lambda\Ical$ contributes with the favourable sign and the Jensen--Ishii lemma for integro-differential equations applies~\cite{BarlesImbert2008}. The standard estimates (Crandall--Ishii--Lions~\cite{CrandallIshiiLions}, \emph{User's Guide}, Thm.~8.2) then drive the penalized maximum to a contradiction as $\eta,\varepsilon\downarrow0$, since $\Pi$ enters \eqref{eq:isaacs} only through first-order (transport) and Lipschitz terms and so is handled exactly as the first-order $Y$-variables. Hence $\underline S\le\overline S$. Uniqueness follows by applying this to two solutions in both orders; Theorem~\ref{thm:existence} identifies the unique solution with the value function.
\end{proof}

\subsection{Semiconcavity, the switching set, and the closed loop}\label{sec:filippov}

\begin{proposition}[Semiconcavity and a.e.\ well-defined feedback]\label{prop:semiconcave}
Assume in addition that each $c_a$ is $C^{1,1}$ in $u$, that $f_a$ is semiconcave in $\hat X$, and that the admissible sets are given by finitely many inequality constraints,
\[
 U_a(\hat X,\Nu)=\{\,u_a:\ g_{a,i}(\hat X,u_a)\le 0,\ i=1,\dots,m_a\,\},
\]
with each $g_{a,i}$ semiconcave in $\hat X$ and convex in $u_a$. Then for each $t$ the value function $S(t,\cdot)$ is semiconcave in the local state variables $Y$, uniformly on compacts. Consequently:
\begin{enumerate}
\item[(a)] $\nabla_Y S(t,\cdot)$ exists Lebesgue-almost everywhere and is of bounded variation, and the singular (kink) set $\Sigma_t=\{w:\,S(t,\cdot)\text{ not differentiable in }Y\}$ is countably rectifiable of dimension $N-1$, hence of zero Lebesgue measure;
\item[(b)] the switching set of the linear components, on which the corresponding coefficient of $\nabla_YS$ vanishes, meets $\Sigma_t$ in a null set, so the bang-bang feedback of Theorem~\ref{thm:nash_feedback}(ii) and the saturated feedback of Theorem~\ref{thm:nash_feedback}(i) are single-valued a.e.;
\item[(c)] the semiconcavity of $S$ in $Y$ makes the closed-loop drift one-sided Lipschitz, so the Filippov/Carath\'eodory closed-loop inclusion $\dot w_t\in \overline{\mathrm{co}}\,F(t,w_t)$ has a unique absolutely continuous solution from each initial condition.
\end{enumerate}
\end{proposition}
\begin{proof}
The $Y$-dynamics \eqref{eq:Yred} are affine in the control with $C^{1,1}$ running cost; semiconcavity of $S(t,\cdot)$ in $Y$ is the value-function semiconcavity theorem for finite-horizon control with semiconcave data and convex velocity sets (\cite{CannarsaSinestrari}, Thm.~7.4.11), the second-order $\hat X$-block contributing a bounded perturbation that preserves semiconcavity uniformly on compacts. (a) is Alexandrov's theorem together with the structure of the singular set of a semiconcave function (Cannarsa--Sinestrari~\cite{CannarsaSinestrari}, Ch.~4). (b): on the differentiability set the transfer Hamiltonian is the piecewise-linear map of Section \ref{sub_opt_ham}, whose kink locus $\{p_{Y_a}=p_{Y_b}\}$ is a Lipschitz hypersurface crossed transversally by the flow a.e.; the exceptional set is null. (c): a semiconcave $S(t,\cdot)$ has $-\nabla_Y S$ monotone (one-sided Lipschitz), and the optimal drift is a monotone (one-sided Lipschitz) map of $w$; Filippov solutions of one-sided Lipschitz inclusions exist and are unique (Filippov; Clarke et al.~\cite{ClarkeLedyaev}).
\end{proof}

\subsection{Viscosity verification}
We continue with the verification when we remove the classical regularity assumption.

\begin{theorem}[Verification without classical regularity]\label{thm:viscverify}
Let $S$ be the value function \eqref{eq:value}, i.e.\ the unique viscosity solution of \eqref{eq:isaacs} (Theorems~\ref{thm:existence}--\ref{thm:comparison}). Let $(\rho^\star,u^\star)(t,w)$ be a measurable selection attaining the pointwise $\min_{(\rho,u)}$ in \eqref{eq:Hbar} for $p=\nabla S$, $M=\nabla^2_{\hat X}S$ at every point of differentiability of $S$, and suppose the closed-loop system admits a solution in the sense of Proposition~\ref{prop:semiconcave}. Then the conclusions (i)--(iii) of Theorem~\ref{thm:verify} hold with this $S$: $S(0,w_0)$ is the optimal Stackelberg value, $\rho^\star$ is an optimal transparency policy, and along the optimal trajectory $u^\star$ is a lower-level Nash equilibrium. No $C^{1,2}$ regularity of $S$ is required.
\end{theorem}
\begin{proof}
The lower bound $S(0,w_0)\le \E\!\int_0^T\ell$ for every admissible $(\rho,u)$ is the subsolution property integrated along trajectories: since $S$ is a viscosity subsolution, the nonsmooth It\^o / Dynkin inequality for viscosity solutions (Lions~\cite{Lions1983}; Fleming--Soner~\cite{FlemingSoner}, Ch.~V; or the doubling argument applied to $S$ and the flow) yields $S(0,w_0)\le \E[\int_0^T\ell\,\diff s + S(T,w_T)]=\E\int_0^T\ell$. The reverse inequality at $(\rho^\star,u^\star)$ uses the supersolution property along the closed-loop flow, on which $S$ is differentiable for a.e.\ $t$ (see Proposition \ref{prop:semiconcave}); there $\partial_tS+\bar H(\cdot,\nabla S,\nabla^2_{\hat X}S)=0$ holds pointwise a.e.\ and the selection attains the Hamiltonian, giving $\E\int_0^T\ell=S(0,w_0)$. Thus the value is attained and $S(0,w_0)$ is optimal, proving (i)--(ii). Claim (iii) is the efficiency collapse: the inner minimizer $u^\star$ is the social optimizer and, under the Groves alignment, each agent's best response, hence a lower-level Nash equilibrium.
\end{proof}

Proposition~\ref{prop:semiconcave} legitimizes the bang-bang feedback of the Nash-strategy theorem: although discontinuous, it is defined a.e., the switching set is negligible, and the closed loop is a well-posed Filippov flow. Together with Theorems~\ref{thm:existence}--\ref{thm:viscverify} this discharges the classical-regularity assumption of Section~\ref{sec:bilevel}: the master value $S$ is the unique viscosity solution of \eqref{eq:isaacs}, the verification conclusions hold without $C^{1,2}$, and the leader's envelope-based first-order condition is evaluated along the a.e.-differentiable optimal trajectory.

\section{Instantiation: Multi-Area Power Systems}\label{sec:power}
This section instantiates $(f_a,U_a,c_a,\pi)$ on multi-area power-system coordination under extreme weather, the setting calibrated in Section~\ref{sec:numerics}.

\paragraph{Agents and local states.} $\Acal$ indexes $N>1$ operational regions connected by cross-border intertie lines $\mathcal{T}$, and $Y_{a,t}\in\R$ is Area $a$'s reserve-shortage state. Writing $\mathcal{G}_a$ for the generating units of Area $a$, $D_a(\hat X_t)$ for its weather-dependent demand, $P_{a,g,t}$ for the dispatch of unit $g$, and $T_{ab,t}$ for the outflow $a$ sends to $b$, the transition map $f_a$ of \eqref{eq:Ydyn} is

\begin{equation}\label{eq:Ydyn_pow}
f_a = D_a(\hat{X}_t) - \!\!\sum_{g \in \mathcal{G}_a}\!\! P_{a,g,t}
- \!\!\sum_{b : (a,b)\in\mathcal{T}}\!\! T_{ab,t} + \!\!\sum_{b : (b,a)\in\mathcal{T}}\!\! T_{ba,t},
\end{equation}

which is affine in the joint control, as \eqref{eq:Ydyn} requires. Area $a$'s control is $u_a=(P_{a,\cdot},T_{a\cdot})$.

\paragraph{Admissible sets.} Generation and intertie capabilities degrade with the observed structural state, capturing icing:
\begin{equation}\label{eq:caps_pow}
\begin{gathered}
0 \le P_{a,g,t} \le P_{a,g,\max}(\hat X_t,\Nu_t), \\
0 \le T_{ab,t} \le \overline{T}_{ab}(\hat X_t,\Nu_t),
\end{gathered}
\end{equation}
so $U_a(\hat X,\Nu)$ of \eqref{eq:caps} is the box \eqref{eq:caps_pow}, which is convex and compact, and Lipschitz in $\hat X$ whenever the envelopes $P_{a,g,\max},\overline T_{ab}$ are. The box is cut out by constraints affine in the control, so the hypothesis of Proposition~\ref{prop:semiconcave} reduces to a condition on the envelopes in $\hat X$. In the calibrated model of Section~\ref{sec:numerics} the envelopes depend on elapsed time and the observed disruption record only, capability being lost when equipment actually trips rather than when severity is merely believed high; $U_a$ is then independent of $\hat X$ and the hypothesis holds trivially. Where belief-dependent envelopes are wanted, it suffices that they be $C^{1,1}$ in $\hat X$. The envelopes depend on the state only through the $\Gfilt_t$-observable pair $(\hat X_t,\Nu_t)$: physical degradation is driven by believed severity and by the observed disruption record, both available to the operators.

\paragraph{Running cost.} The cost $c_a$ of \eqref{eq:cost} is

\begin{equation}\label{eq:cost_pow}
c_a = \sum_{g \in \mathcal{G}_a} C_{a,g}(P_{a,g,t}) + \frac{\Gamma_a}{2} |Y_{a,t}|^2 ,
\end{equation}

with $C_{a,g}$ the production cost of unit $g$, $C^1$, strictly convex and increasing, and $\Gamma_a>0$ the unserved-energy penalty, Area $a$'s Value of Lost Load, weighting the squared shortage. Strict convexity in $P_{a,\cdot}$ and linearity in $T_{a\cdot}$ are exactly the structure Assumption~\ref{ass:data} requires.

\paragraph{Equilibrium form.} Specializing Theorem~\ref{thm:nash_feedback} to \eqref{eq:Ydyn_pow}--\eqref{eq:cost_pow} gives the explicit feedback.
\begin{corollary}[Saturated generation and bang-bang transfers]\label{cor:power_feedback}
Under Assumption~\ref{ass:data} and the instantiation \eqref{eq:Ydyn_pow}--\eqref{eq:cost_pow}, the equilibrium controls are
\begin{align}\label{eq:Pstar}
P_{a,g,t}^* = \max\!\Big( 0, \, \min\big(P_{a,g,\max}(\hat X_t,\Nu_t), (\dot{C}_{a,g})^{-1}(\nabla_Y V_a) \big) \Big),
\end{align}
a saturated soft projection, and
\begin{equation}\label{eq:Tstar}
T_{ab,t}^* = \overline{T}_{ab}(\hat X_t,\Nu_t)\, \mathbf{1}\!\left\{ \nabla_Y V_a < \nabla_Y V_b \right\},
\end{equation}
that is, $a$ sends at full capacity to $b$ precisely when $b$'s marginal shortage value exceeds $a$'s, and nothing otherwise, with arbitrary tie-breaking on the switching surface.
\end{corollary}
\begin{proof}
$P_{a,g}$ enters $c_a$ through the strictly convex $C_{a,g}$, so it is case (i) of Theorem~\ref{thm:nash_feedback} and the stationary point $(\dot C_{a,g})^{-1}(\nabla_YV_a)$ is projected onto $[0,P_{a,g,\max}]$. The outflow $T_{ab}$ carries no own running cost and enters the bracket of \eqref{eq:hjb} linearly through $-\nabla_YV_a\,T_{ab}+\nabla_YV_b\,T_{ab}$, so it is case (ii), with the marginal-value difference $\nabla_YV_b-\nabla_YV_a$ determining the vertex.
\end{proof}

\paragraph{Transfers.} The marginal contribution \eqref{eq:Mdef} is the cost reduction Area $a$'s interties confer on the rest of the system, and the Groves payment \eqref{eq:groves} charges the chief beneficiary and compensates the exporters. Section~\ref{sec:numerics} estimates these payments on the 2021 topology.

\begin{remark}[Dispatch, unit commitment, and modeling scope]\label{rem:opfuc}
The framework operates one layer above intra-area dispatch. Economic dispatch is \emph{embedded}: the strictly convex cost $\sum_g C_{a,g}(P_{a,g})$ is the area's merit-order curve, and the bang-bang inter-area exchange---route power toward the higher shadow price $\partial_Y W$ until the tie binds---is exactly the cross-interface dispatch, read off the value gradients rather than re-solved as a linear program. The network is reduced to scalar ties (loop flow and full DC/AC OPF abstracted away), defensible for ERCOT's scalar-controllable DC ties; a DC-OPF refinement is compatible, inserting a convex quadratic program that turns transfers into OPF flows and the bang-bang rule into a KKT condition while preserving convexity, the efficiency collapse, and semiconcavity. Unit commitment is exogenized as the availability envelope $\Pmax_a(t)$, whose Uri collapse and recovery is the slow commitment state---a boundary condition, not an integer decision. Commitment binaries are excluded deliberately: they would break the convexity on which the viscosity verification, the efficiency collapse (unique social optimum), and the Groves prices (zero duality gap) all rest, making a commitment-aware model a genuinely different (mixed-integer / impulse-control) object. Thus $(C_a,\Pmax_a(t))$ is the pre-solved inner loop; the contribution is the outer loop of coordination, disclosure, and incentive-compatible transfers.
\end{remark}

\subsection{Structural Special Cases}\label{sec:special}
To expose the mechanics transparently, take a scalar latent state and a scalar public signal, $n=m=1$ in \eqref{eq:Xdyn} and \eqref{eq:obs}; two symmetric areas $\mathcal{A}=\{1,2\}$ joined by a single bidirectional tie of capacity $\bar T$; identical quadratic generation cost $C(P)=\tfrac{\varkappa}{2}P^2$; common VOLL $\Gamma$; and time-invariant coefficients, so that $A_t\equiv A$, $\rho_t\equiv\rho$, $\Sigma_t\equiv\Sigma$ and $R$ are scalars.

\begin{proposition}[Closed-form reduction]\label{prop:special}
In the scalar symmetric case:
\begin{enumerate}
\item[(i)] The steady-state belief variance is the stabilizing root of the algebraic Riccati equation $2A\Pi + \Sigma^2 - \tfrac{\rho^2}{R}\Pi^2 + \lambda\Sigma_J = 0$, namely
\begin{equation}\label{eq:Piinf}
\Pi_\infty = \frac{R}{\rho^2}\Big( A + \sqrt{A^2 + \tfrac{\rho^2}{R}(\Sigma^2+\lambda\Sigma_J)} \Big),
\end{equation}
strictly decreasing in the disclosure intensity $\rho^2/R$: more transparency lowers belief variance.
\item[(ii)] Generation is the soft projection $P^*_a = \min\!\big(P_{\max}(X),\,\max(0,\nabla_Y V_a / \varkappa)\big)$.
\item[(iii)] The tie flows from the area with the lower marginal shortage value to the higher one at full capacity $\bar T$ whenever $\nabla_Y V_1 \neq \nabla_Y V_2$, and is zero on the symmetric manifold $\{Y_1=Y_2\}$, which is exactly the switching surface.
\end{enumerate}
\end{proposition}
\begin{proof}
(i) is the scalar specialization of \eqref{eq:riccati} with $\dot\Pi=0$, taking the positive (stabilizing) root. (ii)--(iii) specialize Theorem~\ref{thm:nash_feedback} with $(\dot C)^{-1}(y)=y/\varkappa$ and symmetric value functions $V_1(\cdot,Y_1,Y_2)=V_2(\cdot,Y_2,Y_1)$, so $\nabla_Y V_1=\nabla_Y V_2$ iff $Y_1=Y_2$.
\end{proof}

A second structural consequence, illustrated in Experiment~2 of Section~\ref{sec:numerics}, concerns the interaction of disclosure and coupling under \emph{systemic} risk. Consider $N$ renewable-heavy areas whose net-load deviations follow a one-factor structure $N_a=b_a F+\zeta_a$ with a common factor $F$ and independent idiosyncratic terms $\zeta_a$; disclosure reveals $F$ and lowers the common-mode belief variance $v_F(\rho)=P_{FF}(\rho)$, while each area carries precautionary reserve at unit cost $g(z^\star)$.

\begin{proposition}[Information--coupling complementarity under common risk]\label{prop:compl}
Suppose net loads admit the one-factor structure $N_a=b_a F+\zeta_a$ with loadings $b_a\ge 0$ and mutually independent idiosyncratic terms, and that disclosure reduces only the common-mode belief variance $v_F(\rho):=P_{FF}(\rho)$, strictly decreasing in $\rho$, while the idiosyncratic forecast variances $\sigma_{\zeta,a}^2$ are fixed. Write $\sigma_a^2 = v_F b_a^2 + \sigma_{\zeta,a}^2$ and $\sigma_\Sigma^2 = v_F\big(\sum_a b_a\big)^2 + \sum_a \sigma_{\zeta,a}^2$, and let the pooling benefit be $B(\rho) = g(z^\star)\big(\sum_a \sigma_a(\rho) - \sigma_\Sigma(\rho)\big)\ge 0$. Then
\begin{enumerate}
\item[(i)] $\displaystyle\lim_{v_F\to\infty} B = 0$: perfectly correlated common risk is non-diversifiable;
\item[(ii)] $\displaystyle\lim_{v_F\to 0} B = g(z^\star)\Big(\sum_a \sigma_{\zeta,a} - \sqrt{\textstyle\sum_a \sigma_{\zeta,a}^2}\Big) > 0$ (strict unless all but one $\sigma_{\zeta,a}$ vanish);
\item[(iii)] with $\sigma_a=\sqrt{v_F b_a^2+\sigma_{\zeta,a}^2}$ and $\sigma_\Sigma=\sqrt{v_F(\sum_a b_a)^2+\sum_a\sigma_{\zeta,a}^2}$, the pooling benefit is locally increasing in disclosure ($\partial B/\partial v_F<0$) if and only if
\begin{equation}\label{eq:monocond}
\sum_a \frac{b_a^2}{\sigma_a} \;<\; \frac{\big(\sum_a b_a\big)^2}{\sigma_\Sigma}.
\end{equation}
\end{enumerate}
By (i)--(ii) the value of market coupling is strictly larger under full disclosure ($v_F\to0$) than under no disclosure ($v_F\to\infty$); the value of disclosure is correspondingly larger under coupling than under autarky, by exactly $B(v_F^{\mathrm{lo}})-B(v_F^{\mathrm{hi}})>0$ for $v_F^{\mathrm{lo}}<v_F^{\mathrm{hi}}$. Thus information design and market coupling are complements. 
\end{proposition}
\begin{proof}
Substitute the factor decomposition. As $v_F\to\infty$, $\sigma_a\sim\sqrt{v_F}\,b_a$ and $\sigma_\Sigma\sim\sqrt{v_F}\sum_a b_a$ (using $b_a\ge0$), so $\sum_a\sigma_a-\sigma_\Sigma\to0$, giving (i). As $v_F\to0$, $\sigma_a\to\sigma_{\zeta,a}$ and $\sigma_\Sigma\to\sqrt{\sum_a\sigma_{\zeta,a}^2}$, and $\sum_a\sigma_{\zeta,a}\ge\sqrt{\sum_a\sigma_{\zeta,a}^2}$ (with equality only in the degenerate case) gives (ii). For (iii), differentiate: $\partial_{v_F}\sigma_a=b_a^2/(2\sigma_a)$ and $\partial_{v_F}\sigma_\Sigma=(\sum_a b_a)^2/(2\sigma_\Sigma)$, hence $\partial_{v_F}B=g(z^\star)\big(\sum_a b_a^2/(2\sigma_a)-(\sum_a b_a)^2/(2\sigma_\Sigma)\big)$, which is negative iff \eqref{eq:monocond} holds. The cross-regime comparative static follows from $\mathrm{Cost}_{\text{unc}}-\mathrm{Cost}_{\text{cpl}}=B$: the difference of disclosure values between the coupled and autarkic regimes equals $B(v_F^{\mathrm{lo}})-B(v_F^{\mathrm{hi}})$, which is positive by (i)--(ii) alone and requires no monotonicity in between.
\end{proof}

Condition~\eqref{eq:monocond} holds throughout the calibrated sweep of Section~\ref{sub:exp2} (Figure~\ref{fig:exp2_cost}, where the shaded gap widens), so the complementarity is monotone there; it can fail when the loadings are highly concentrated (a single $b_a$ dominating) or the idiosyncratic scales are extremely heterogeneous---precisely when the ``common'' factor is effectively idiosyncratic---in which case $B$ is non-monotone in $v_F$ while remaining nonnegative.

\section{Numerical Illustration of the Theoretical Results}\label{sec:numerics}
We report two complementary experiments. Experiment~1 (Section~\ref{sec:hjbsolve}) solves the master HJB equation under a \emph{real-data calibration}---FERC/NERC event reports, ERCOT's published DC-tie capacities, and the realized February 2021 Winter Storm Uri trajectory---stressing the \emph{physical} and \emph{transfer} channels with the marginal-contribution mechanism active; a sensitivity analysis sweeps the decisive structural parameter, the ERCOT--SPP tie capacity, from the actual near-islanded $0.82$~GW to the FERC/DOE-recommended range. Experiment~2 (Section~\ref{sub:exp2}) isolates the \emph{information} channel: calibrated to documented European renewable-drought statistics, physical caps are removed so that the binding constraint is the economic cost of carrying precautionary reserve sized to forecast uncertainty, allowing the welfare value of disclosure to be measured cleanly.

\subsection{Methodology}\label{sub:method}
Experiment~1 solves the master HJB equation on a $240\times240$ grid over the reserve states and evaluates the resulting feedback by forward simulation along the realized Uri demand and outage trajectory. \emph{Both} levers are active. The storm severity $X_t$ is latent, mean-reverting to the realized Uri profile, and is not observed by the operators; they learn it only through the public advisory \eqref{eq:obs}, whose gain $\rho$ is the leader's disclosure control. Storm onset is a publicly observed disruption epoch, so the belief variance $\Pi_t$ follows the Riccati flow \eqref{eq:riccati} with the reset of Lemma~\ref{lem:proj_filter} at that epoch (Assumption~\ref{ass:observed_jumps}). Because area demand is affine in severity, $D_a = D_{a,0} + \Delta D_a X_t$, the \emph{unresolved} forecast error enters the reserve balance \eqref{eq:Ydyn} as an additional diffusion of variance $\Delta D_a^2\,\Pi_t$: writing $\sigma_a(t)^2 = \sigma_0^2 + \Delta D_a^2\,\Pi_t(\rho)$, disclosure acts on welfare by shrinking $\Pi_t$ and hence the balancing risk the operators must carry. This is the exact channel of Section~\ref{sec:master}, and it is what makes $\rho$ a control input rather than a parameter. Marginal contributions $M_a$ are estimated by a leave-one-out (empirical Groves) construction under common random numbers, evaluated at the net-optimal disclosure level. Experiment~2 removes physical caps; generation is elastic and the operator schedules a precautionary reserve $r = z^\star \sigma$, where $\sigma$ is the net-load forecast standard deviation obtained from the steady-state belief covariance and $z^\star$ is the cost-optimal reliability quantile. Its cost curves are computed analytically from the steady-state Riccati and the Gaussian reserve-loss function, so they are exact rather than Monte-Carlo estimates; where sample trajectories are shown, the latent factor structure is simulated by Euler--Maruyama with beliefs from the exact filter of Lemma~\ref{lem:proj_filter} at the steady Kalman--Bucy gain.

\subsection{Experiment 1: HJB Solution under a Real-Data (Winter Storm Uri) Calibration}\label{sec:hjbsolve}
We solve the master HJB equation numerically and calibrate the three-area
system to the February 2021 Winter Storm Uri event (ERCOT, SPP, MISO). The solve
replaces the first-order certainty-equivalent closure used in the analytic
sections with a genuine value-function solution, and it numerically verifies the
structural predictions of the viscosity section (semiconcavity; a measure-zero
bang-bang switching set).

\subsubsection{Calibration to Winter Storm Uri}\label{app:numcalib}
Parameters are grounded in the FERC/NERC joint report, ERCOT and SPP/MISO winter
event reports, and ERCOT's published DC-tie capacities. The decisive structural
fact is that ERCOT is electrically near-islanded: its total external transfer
capability is about $1.1$ GW (two DC ties to SPP totalling $0.82$ GW plus
$\sim\!0.4$ GW to Mexico), i.e.\ under $2\%$ of its $\sim\!70$ GW winter peak. The calibration parameters and their sources are collected in Table \ref{table_uri_cal}.
\begin{table*}[ht]\label{table_uri_cal}\centering
\small
\begin{tabular}{lcccc}
\toprule
Quantity & ERCOT & SPP & MISO & Source \\
\midrule
pre-storm demand (GW)        & 55 & 32 & 75  & ERCOT/SPP archives \\
storm-peak demand (GW)       & 77 & 46 & 100 & 76.8 GW projected ERCOT peak \\
available capacity (GW)      & 83 & 55 & 130 & SARA / winter capacity \\
capacity forced offline (GW) & 34 & 12.5 & 8 & $>$30\% ERCOT; gas $27$ GW out \\
VOLL proxy $\Gamma$ (\$k/MWh) & 9 & 4 & 3.5 & \$9000 cap; SPP RT \$4029 \\
\midrule
\multicolumn{5}{l}{Ties: ERCOT--SPP $=0.82$ GW (DC); SPP--MISO $\approx 10$ GW (synchronous); ERCOT--MISO $=0$.}\\
\bottomrule
\end{tabular}

\caption{Uri calibration. Storm onset day 3, trough days 5--7, recovery by day 9.5, horizon 12 days.}\label{tab:uri}
\end{table*}

The disclosure layer requires a calibration of the \emph{uncertainty} about severity, which is not a parameter of the physical event but of what was known about it. It can be read directly off ERCOT's own planning miss. The extreme-weather scenario in ERCOT's Seasonal Assessment of Resource Adequacy put the winter peak at $67.2$~GW, while the realized peak demand absent load shedding was $76.8$~GW~\cite{USHouse2021}: a $9.6$~GW error against the operator's own worst case. Since demand is affine in severity with $\Delta D_E = 22$~GW (Table~\ref{tab:uri}), the extreme scenario corresponds to a severity of $0.56$ and the realization to $1.0$, so the prior severity error was $0.44$ and the prior variance is $\Pi_0 = 0.44^2 = 0.191$. We take the severity to be an Ornstein--Uhlenbeck process reverting to the realized Uri profile, with process volatility chosen so that its stationary variance equals $\Pi_0$. The remaining belief-channel parameters are collected in Table~\ref{tab:uri_info}; the disclosure price $\Lambda$ is set so that the optimum is interior, as the reserve-holding cost is in Experiment~2.

\begin{table}[h]\centering\small
\caption{Experiment 1: severity-uncertainty and disclosure calibration. $\Pi_0$ is fixed by ERCOT's $9.6$~GW miss against its own extreme-weather scenario.}\label{tab:uri_info}
\resizebox{0.7\textwidth}{!}{%
\begin{tabular}{lcl}
\toprule
Parameter & Value & Meaning \\
\midrule
$\Pi_0$       & $0.191$ & prior severity variance (from the $9.6$ GW miss) \\
$\sigma_X$    & $0.437$ & severity volatility (stationary variance $=\Pi_0$) \\
$\theta$      & $0.50$  & severity mean reversion \\
$R$           & $1.0$   & advisory observation noise \\
$\Sigma_J$    & $0.05$  & covariance reset at storm onset \\
$\sigma_0$    & $2.0$   & residual balancing noise \\
$\Lambda$     & $40$    & disclosure price \\
\bottomrule
\end{tabular}%
}
\end{table}

\begin{figure}\centering
    \includegraphics[width=\linewidth]{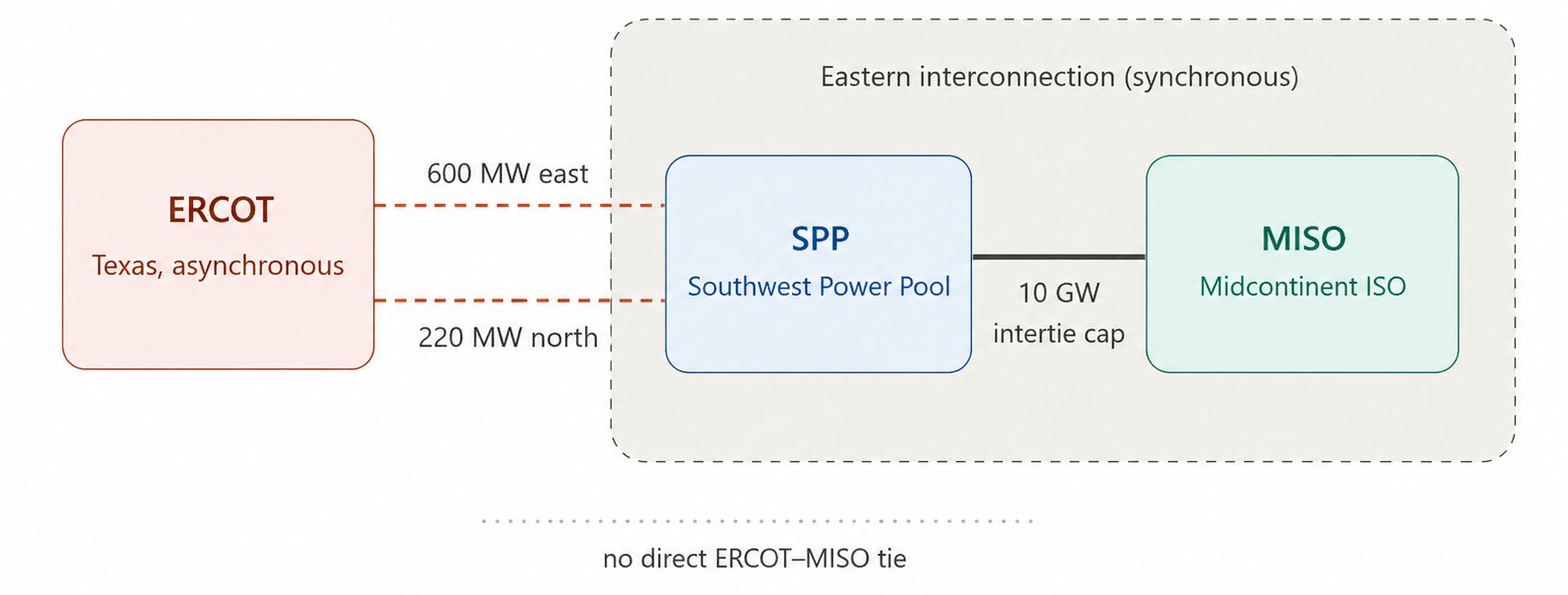}
    \caption{Interconnection topology of the calibrated three-area system. ERCOT is asynchronous: its only links to the Eastern interconnection are two back-to-back HVDC ties, $220$~MW at Oklaunion and $600$~MW at Monticello--Welsh, both terminating in SPP. SPP and MISO are synchronously connected, with transfer capability limited by congestion rather than by a converter rating. There is no direct ERCOT--MISO tie, so mutual aid from MISO reaches ERCOT only through SPP}
    \label{fig:topology}
\end{figure}

The topology is shown in Figure~\ref{fig:topology}. We reduce to the ERCOT--SPP core (the two DC-tied, capacity-constrained areas)
and fold MISO in as an external reserve feeding SPP through the large synchronous
interface, available up to $\min(10\,\text{GW},\,\text{MISO spare})$ at an import
price $p_M=1.5$. This keeps the value problem two-dimensional and the switching
set visualizable while preserving the mutual-aid chain MISO$\to$SPP$\to$ERCOT.

\subsubsection{Solving the HJB equation}

By the efficiency-collapse theorem, the Groves-aligned lower-level Nash
equilibrium coincides with the social optimum, so we solve a \emph{single}
social-planner value $W(t,Y)$ over reserve states $Y=(Y_E,Y_S)$ rather than $N$
coupled best-response functions. With the optimal generation and bang-bang
transfer substituted, the exact optimized Hamiltonian is
\begin{align}
\bar H(Y,p,t)&=\tfrac12\Gamma_E Y_E^2+\tfrac12\Gamma_S Y_S^2
\nonumber\\&\quad+p_E D_E(t)+p_S D_S(t)
\nonumber\\&\quad+G(p_E)+G(p_S)-\tau_{ES}\,|p_E-p_S|
\nonumber\\&\quad-\big(p_S-p_M\big)^+ I_{\max}(t),
\end{align}

where $G(p)=\min_{0\le P\le \Pmax}\{\tfrac12 cP^2-pP\}$ is the saturated-generation
term and $-\tau_{ES}|p_E-p_S|$ is the bang-bang inter-area exchange. We integrate
$\partial_\tau W=\bar H$ backward from $W(T,\cdot)=0$ with a monotone
Lax--Friedrichs scheme,
\begin{align*}
W^{k+1}&=W^k+\Delta\tau\Big(\bar H(Y,p^0)+\tfrac{\rho_E}{2}(p_E^+-p_E^-)\\
&\quad+\tfrac{\rho_S}{2}(p_S^+-p_S^-)\Big),\quad
\rho_a\ge\max|\partial_{p_a}\bar H|,
\end{align*}
on a $240\times240$ grid over $[-15,200]$ GW$\cdot$day, with a constant-gradient
outflow boundary condition. The scheme is monotone, consistent and stable under
$\Delta\tau\,(\rho_E+\rho_S)\le \Delta y$, hence converges to the unique
viscosity solution (Barles--Souganidis~\cite{BarlesSouganidis}), consistent with the viscosity-verification
section. Generation feedback is $P_a^\star=\mathrm{clip}(\partial_{Y_a}W/c_a,0,\Pmax_a)$
and the transfer routes full tie capacity toward the higher marginal value.

\subsubsection{Results}

\paragraph{The disclosure lever.}
Sweeping the disclosure gain over $\rho\in[0.2,4.0]$ drives the belief variance at the storm trough from $\Pi=0.192$ to $\Pi=0.082$, and social cost falls monotonically: by $8.7\%$ under market coupling ($150{,}439\to137{,}350$) and by $7.8\%$ under autarky ($161{,}170\to148{,}616$). Net of the disclosure price the optimum is interior at $\rho^\star=3.0$ in both regimes (Fig.~\ref{fig:uri_disclosure}). Two things follow. First, the pooling benefit \emph{widens} with disclosure, from $10{,}731$ to $11{,}266$, so the two levers are complements on the Uri topology exactly as Proposition~\ref{prop:compl} predicts under a common factor. Second, the disclosure gain is an order of magnitude smaller here than in the common-risk experiment of Section~\ref{sub:exp2}, where the same lever is worth $37\%$ and $48\%$. This is the quantitative content of the regime distinction: when the binding constraint is a physical deficit of tens of gigawatts, better information cannot manufacture capacity, and the marginal value of transparency is second order; when the binding constraint runs through forecast uncertainty, it is first order. The two experiments measure the same controller in the two regimes.

\paragraph{Full solve versus the first-order closure.}
At $\rho^\star$, forward simulation under the solved feedback gives a realized social cost of $1.395\times10^5$, against $1.548\times10^5$ for the first-order closure $\partial_{Y_a}V_a\approx\Gamma_a Y_a^+$ and $1.507\times10^5$ for autarky. The full solve improves on the closure by $\mathbf{9.9\%}$ and on autarky by $\mathbf{7.4\%}$, and reduces energy not served by $4.8\%$ relative to autarky (Fig.~\ref{fig:welfare}).

\paragraph{Coupling under islanding.}
The $7.4\%$ coupling gain is small precisely because of ERCOT's $0.82$~GW tie: mutual aid cannot relieve a $\sim\!25$~GW deficit through a sub-GW interface.

\begin{figure*}[ht]\centering
\includegraphics[width=0.92\textwidth]{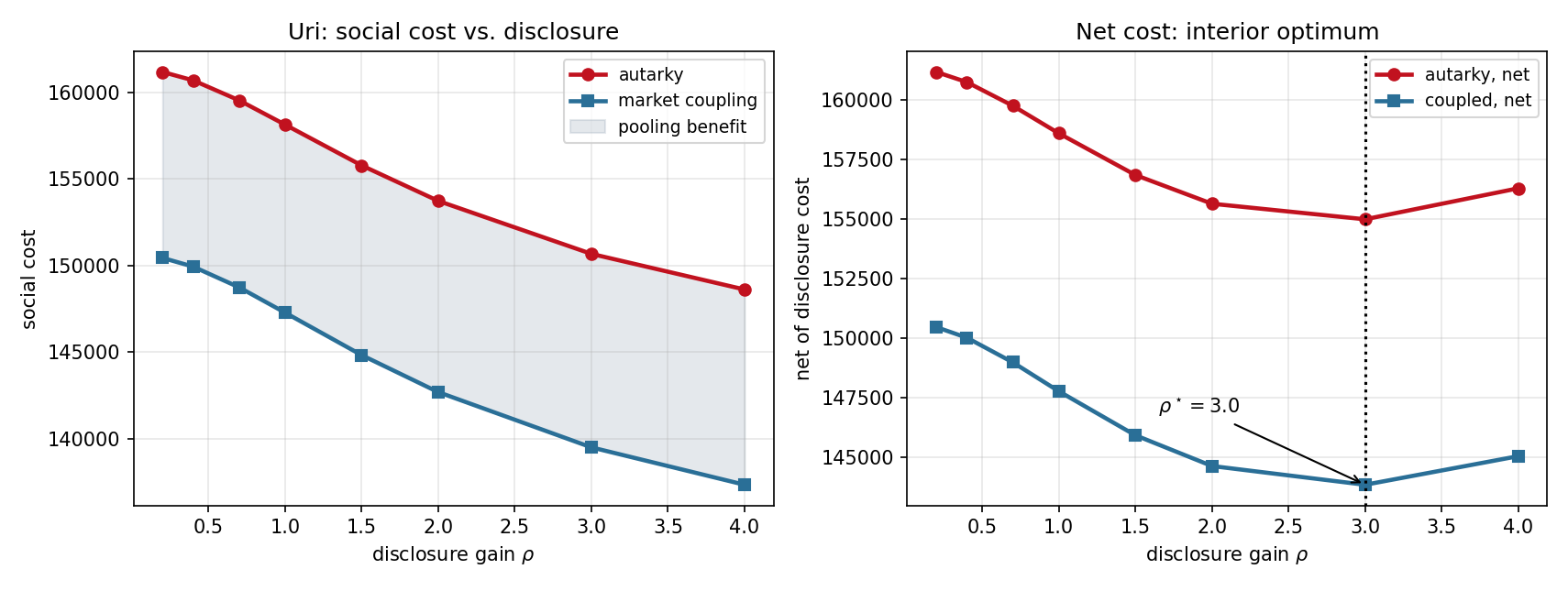}
\caption{Experiment 1: the information lever on the Uri calibration. Left: social cost against the disclosure gain, under autarky and under market coupling; the shaded pooling benefit widens with disclosure, so the two levers are complements. Right: net of the disclosure price, the optimum is interior at $\rho^\star=3.0$.}
\label{fig:uri_disclosure}
\end{figure*}

\begin{figure*}[ht]\centering
\includegraphics[width=0.92\textwidth]{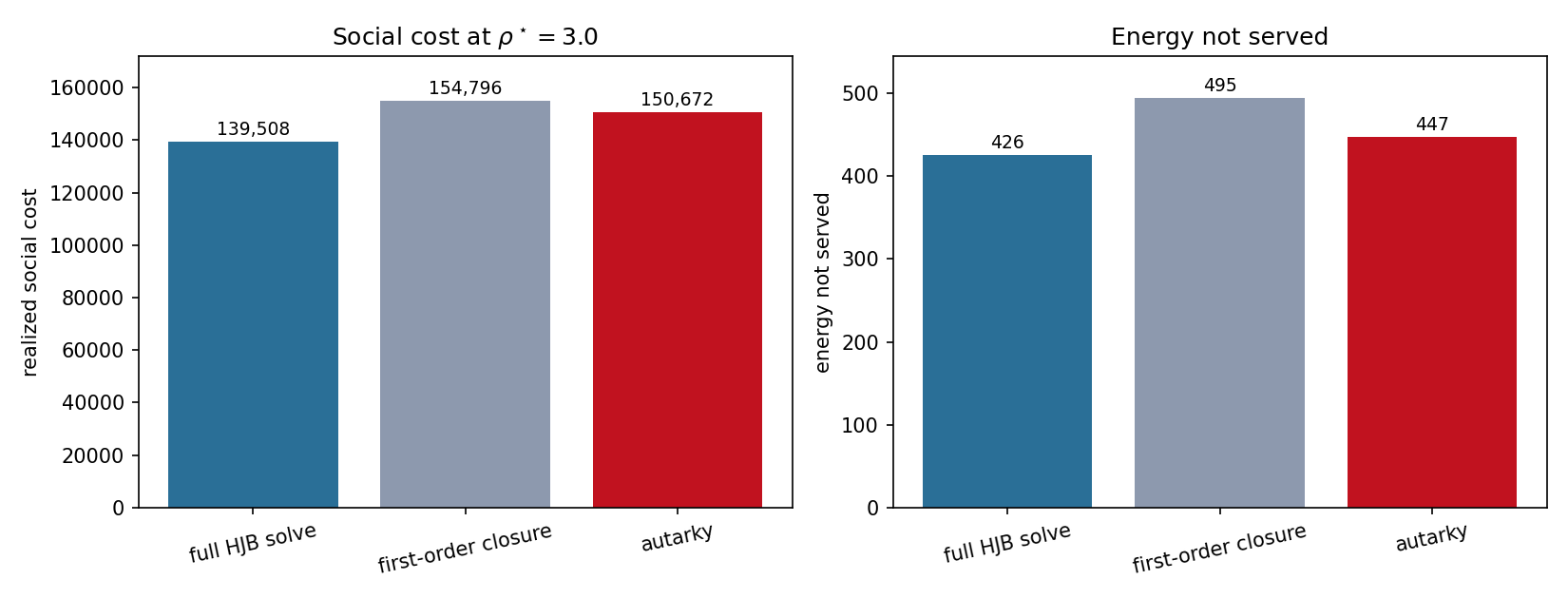}
\caption{Experiment 1: realized social cost and energy not served at the optimal disclosure level: full HJB solve vs.\ first-order closure vs.\ autarky, under the Uri calibration.}
\label{fig:welfare}
\end{figure*}

\paragraph{Groves transfers on the real topology.}
The efficiency collapse solves the planner problem without transfers, but the transfers themselves are recovered from leave-one-out marginal contributions at $\rho^\star$, estimated under common random numbers: area $a$'s export capability is removed and the increase in the remaining areas' realized cost is recorded. MISO's marginal contribution is $M_{\mathrm{MISO}}\approx10.1$k (removing its imports raises the coupled cost from $1.395\times10^5$ to $1.496\times10^5$); SPP's is $M_{\mathrm{SPP}}\approx11.3$k (severing the ERCOT--SPP tie raises ERCOT's cost from $1.355\times10^5$ to $1.468\times10^5$); and the stricken importer contributes essentially nothing, $M_{\mathrm{ERCOT}}\approx-0.18$k, its presence slightly \emph{raising} SPP's cost through the export burden. The Groves payments therefore compensate the exporters and charge the chief beneficiary, exactly as Theorem~\ref{thm:ic} prescribes: SPP receives its marginal contribution, which exceeds its $0.18$k participation burden by two orders of magnitude, making participation individually rational, while ERCOT is charged. Two structural features are visible. First, the sub-GW interface caps not only the welfare gain but the mechanism's transfer volume. Second, the marginal contributions do not sum to the total coupling surplus ($\approx11.2$k), the gap being $\approx10.1$k, which is the numerical face of the Green--Laffont budget imbalance of Section~\ref{sec:gl}.

\subsubsection{Sensitivity analysis: interregional transfer capability}\label{sub:sens}
The decisive structural parameter of the calibration is the ERCOT--SPP tie capacity, fixed by the realized 2021 topology at $0.82$~GW. Sweeping it toward the FERC/DOE recommended $20$--$25\%$ of peak, and holding disclosure at $\rho^\star$, social cost falls monotonically, by about
$35\%$ and energy not served by about $28\%$ at $10$ GW ($13\%$ of peak), with sharply
diminishing returns beyond $\sim\!8$--$10$ GW (Fig.~\ref{fig:tie}). The
saturation has a physical reading: the marginal value of ERCOT interconnection is
bounded by its neighbors' \emph{simultaneous} scarcity, since SPP and MISO are
themselves capacity-short during the same correlated event. Correlated regional
stress thus caps the benefit of point-to-point ties---an argument the present
information/common-factor framework is built to express. The comparison of the full solve against the first-order closure above serves as a solution-method sensitivity check: the reported magnitudes are properties of the calibrated problem, not artifacts of the closure. The tie sweep is run at the net-optimal disclosure level, so the transfer and information levers are jointly at their optima throughout.

\begin{figure*}[ht]\centering
\includegraphics[width=0.72\textwidth]{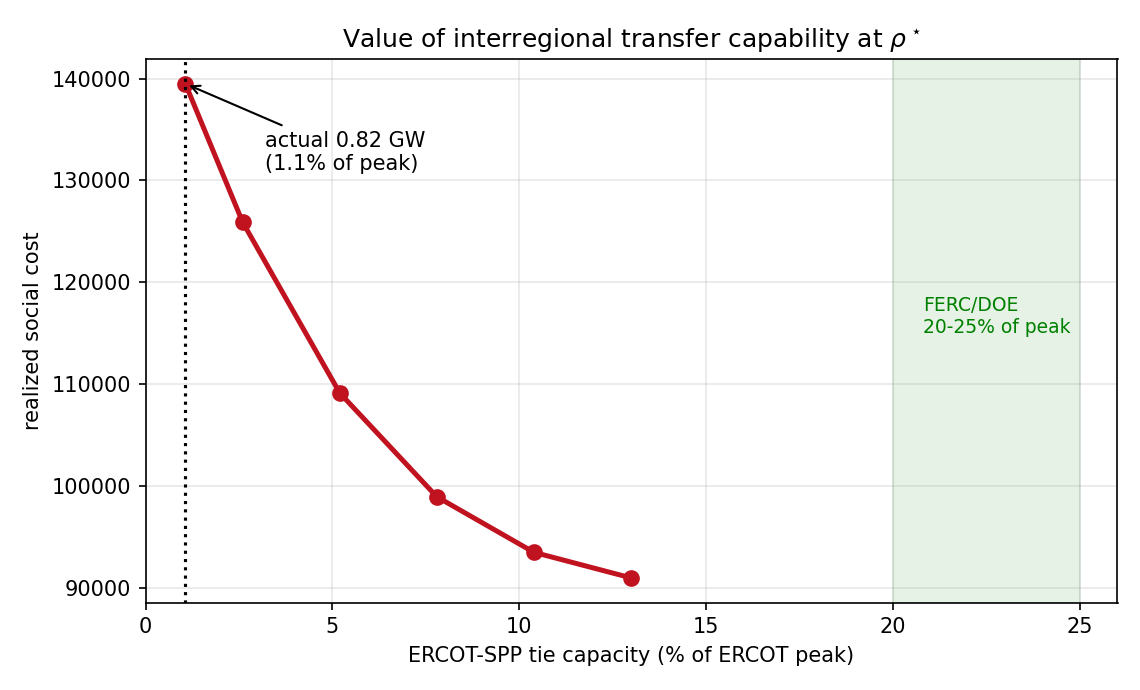}
\caption{Experiment 1: Value of interregional transfer capability: realized social cost vs.\
ERCOT--SPP tie capacity. Actual $0.82$ GW marked; FERC/DOE $20$--$25\%$ band shaded.}
\label{fig:tie}
\end{figure*}

\paragraph{Numerical verification of the viscosity theory.}
The solved value function is semiconcave: the curvature $\partial^2_{Y_E}W$ is
bounded above (Fig.~\ref{fig:value}, right), and the bang-bang switching locus
$\{\partial_{Y_E}W=\partial_{Y_S}W\}$ is a one-dimensional curve---a Lebesgue-null
set in the state plane (Fig.~\ref{fig:value}, left)---exactly as the semiconcavity
proposition predicts. The optimal feedback is therefore single-valued almost
everywhere and the closed loop is a well-posed Filippov flow.

\begin{figure*}[ht]\centering
\includegraphics[width=\textwidth]{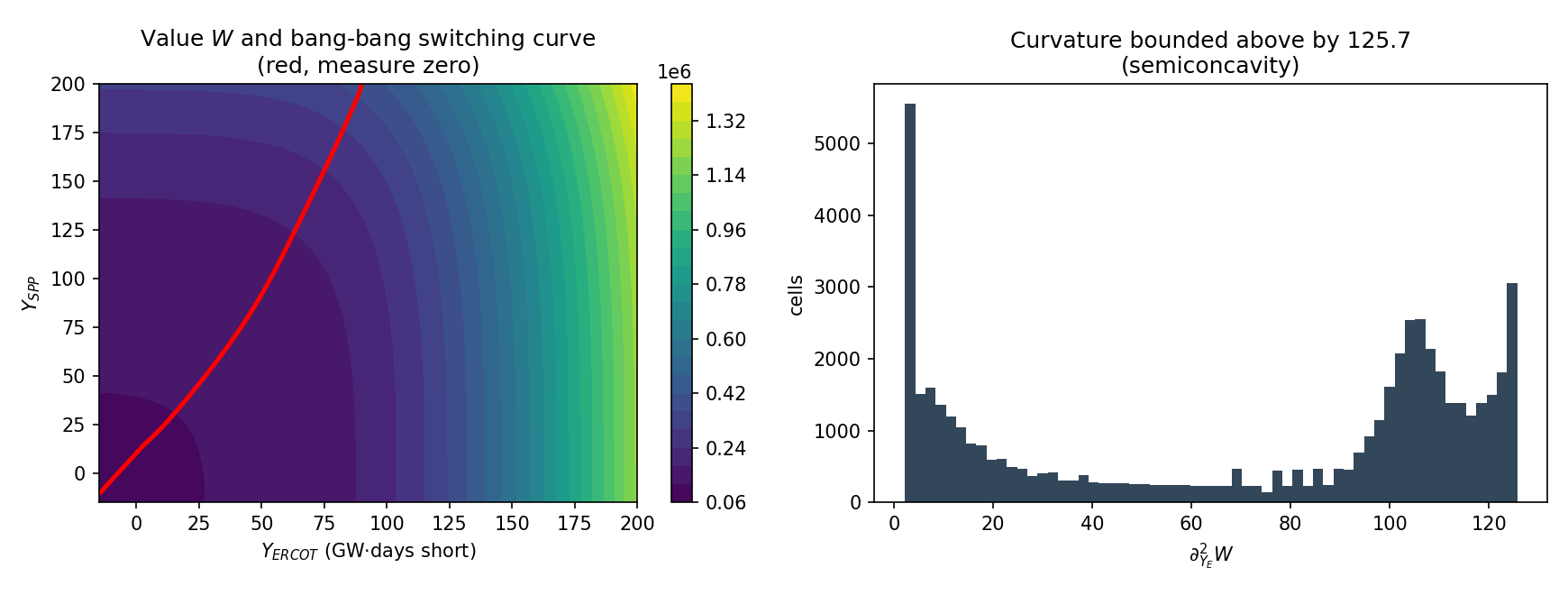}
\caption{Experiment 1. Left: solved value $W$ and the bang-bang switching curve
$\partial_{Y_E}W=\partial_{Y_S}W$ (measure zero). Right: curvature distribution,
bounded above (semiconcavity).}
\label{fig:value}
\end{figure*}

\subsection{Experiment 2: The Value of Disclosure under Common Risk}\label{sub:exp2}
We now consider three renewable-heavy zones (Iberia/solar, North~Sea/wind, Nordics/hydro) under a common continental weather factor (``Dunkelflaute''). Net-load deviations follow a one-factor structure $N_a = b_a F + \zeta_a$, with $F$ and $\zeta_a$ independent Ornstein--Uhlenbeck processes. There are no physical caps: generation is elastic, and each area must carry precautionary reserve $r=z^\star\sigma$ to meet a reliability target, at holding cost $\kappa_r$ per unit and shortfall penalty $\Gamma$. A public advisory reveals the systemic factor $F$ with gain $\rho$ (the disclosure lever), while each area has a fixed local sensor of its own $\zeta_a$. The steady belief covariance $P(\rho)$ solves the continuous algebraic Riccati equation, and disclosure lowers the common-mode variance $P_{FF}(\rho)$. Calibration is in Table~\ref{tab:exp2}.

The calibration reflects documented European \\``Dunkelflaute'' statistics ($1.6$ events/year, Nov--Jan, highest exposure in correlated northern offshore-wind fleets, cross-border integration cutting drought frequency only $\sim\!9\%$)--the un-poolable common-mode signature the disclosure layer must address.

\begin{table}[h]
\centering
\small
\caption{Experiment 2 calibration (latent state $Z=[F,\zeta_1,\zeta_2,\zeta_3]$).}\label{tab:exp2}
\begin{tabular}{lcccc}
\toprule
Parameter & $F$ & Iberia & North Sea & Nordics \\
\midrule
Mean reversion $\kappa$        & 0.40 & 0.80 & 0.80 & 0.80 \\
Process volatility $\sigma$    & 1.00 & 0.60 & 0.90 & 0.50 \\
Factor loading $b_a$           & --   & 0.70 & 1.20 & 0.50 \\
\bottomrule
\end{tabular}
\\[2pt]
\footnotesize Local sensor precision $\gamma{=}0.8$; reserve-holding cost $\kappa_r{=}2$, shortfall penalty $\Gamma{=}12$, giving optimal reserve quantile $z^\star{=}0.967$ and unit uncertainty cost $g(z^\star){=}2.998$; observation noise $R{=}I$; disclosure swept $\rho\in[0.2,4.0]$.
\end{table}

As disclosure rises from $\rho=0.2$ to $4.0$, the common-mode variance falls from $P_{FF}=1.18$ to $0.23$, and expected uncertainty cost falls monotonically (Figure~\ref{fig:exp2_cost}): by $37.4\%$ under autarky ($9.03\to5.66$) and by $47.5\%$ under coupling ($8.26\to4.33$). The larger reduction under coupling is the central qualitative finding. Because a common factor cannot be diversified by pooling reserves, market coupling on its own captures only the idiosyncratic share of the risk; it is \emph{disclosure} that attacks the systemic common mode. Consequently the value of coupling---the pooling benefit, shaded in Figure~\ref{fig:exp2_cost}---\emph{grows} with disclosure, from $0.77$ to $1.32$. This complementarity is the content of Proposition~\ref{prop:compl} (Section~\ref{sec:power}): under a common factor, information design and market coupling are complements, and the monotonicity condition~\eqref{eq:monocond}---which holds throughout this calibrated sweep, where the shaded gap in Figure~\ref{fig:exp2_cost} widens---governs when the effect is monotone in the disclosure level.

\begin{figure*}[ht]
\centering
\includegraphics[width=0.74\textwidth]{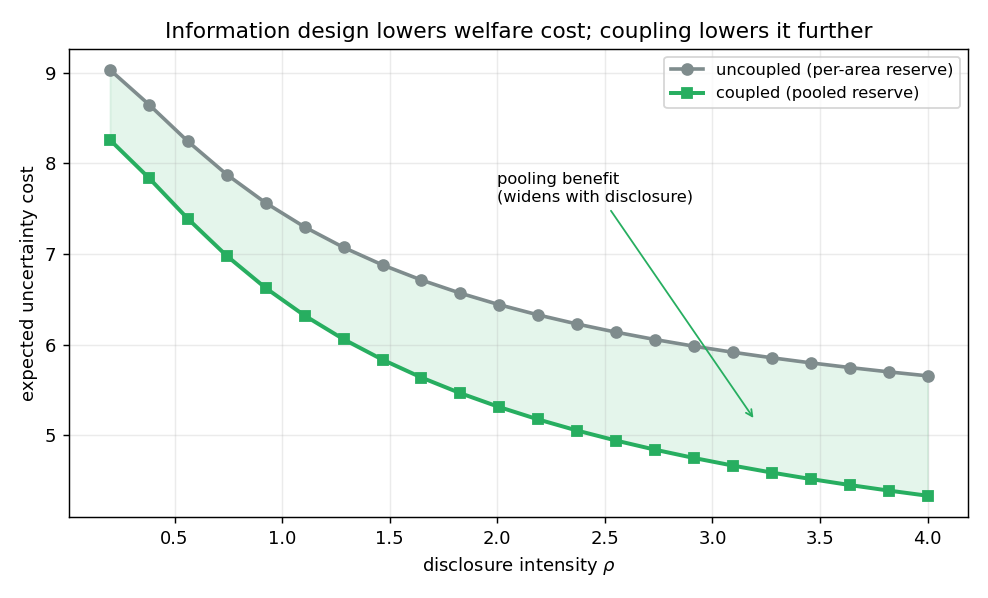}
\caption{Experiment 2. Expected uncertainty cost vs.\ disclosure intensity. Both regimes decrease monotonically; the shaded gap is the pooling benefit, which widens with disclosure.}\label{fig:exp2_cost}
\end{figure*}

\subsection{Discussion}\label{sub:disc}
Together the experiments separate the two levers of the master problem~\eqref{eq:master} while running both. Experiment~1 activates the transfer channel and the information channel simultaneously on the actual 2021 topology. It shows that ERCOT's near-islanded $0.82$~GW interconnection sharply caps the coupling benefit at $7.4\%$; that the tie-capacity sensitivity analysis quantifies the welfare a FERC/DOE-scale interface would unlock, about $35\%$ by $10$~GW with saturation beyond; that the solved value function numerically verifies the semiconcavity and measure-zero switching set predicted by the viscosity theory; and that the leave-one-out Groves estimates rank the areas as economics predicts, compensating the well-resourced exporters and charging the stricken importer. It also prices the information lever in that regime: with the belief calibrated to ERCOT's own $9.6$~GW miss against its extreme-weather scenario, disclosure is worth $8.7\%$ of social cost, with an interior optimum at $\rho^\star=3.0$.

Experiment~2 places the same controller in the regime where forecast uncertainty, rather than physical capacity, is the binding margin, and there the same lever is worth $37\%$ under autarky and $48\%$ under coupling. The contrast is the point. The transparency matrix $\rho_t$ is the regulator's information-design lever, controlling $\Pi_t$ through \eqref{eq:riccati} in both experiments; its welfare value is an order of magnitude smaller when a physical deficit of tens of gigawatts binds, because information cannot manufacture capacity, and first order when the deficit runs through belief. The complementarity survives in both: on Uri the pooling benefit widens from $10{,}731$ to $11{,}266$ as disclosure rises, and in the common-risk calibration the shaded gap of Figure~\ref{fig:exp2_cost} widens likewise. Under a common factor the two instruments are complements, so they should not be treated as substitutes; condition~\eqref{eq:monocond} says when this is monotone in the disclosure level, and fails when the factor is concentrated on one area and hence not truly systemic.

\section{Conclusion}\label{sec:conclusion}
We have posed information control of a differential game as a control problem and shown that adjoining a transfer rule makes it tractable. Incentive alignment turns the lower level into a potential game, so the bilevel problem collapses to a single stochastic control problem (Theorem~\ref{thm:collapse}) and the leader's first-order condition is exact by Danskin's theorem (Theorem~\ref{thm:verify}), a simplification that holds because of the mechanism rather than despite the bilevel structure. The lower level is well posed in its own right: the equilibrium feedback is saturated on strictly convex components and bang-bang on linear ones (Theorem~\ref{thm:nash_feedback}), and exists and is unique up to null modification (Theorem~\ref{thm:nash_exist}), with truthful reporting dominant and efficient action the induced Nash response (Theorem~\ref{thm:ic}).

Under publicly observed disruption epochs the belief filter is exactly finite-dimensional (Lemma~\ref{lem:proj_filter}), so the master value is well defined on a finite-dimensional state and satisfies an integro-differential Hamilton--Jacobi--Bellman equation (Theorem~\ref{thm:master_robust}). That equation admits a unique viscosity solution (Theorems~\ref{thm:existence} and~\ref{thm:comparison}), verification holds without smoothness (Theorem~\ref{thm:viscverify}), and semiconcavity makes the switching set Lebesgue-null (Proposition~\ref{prop:semiconcave}), so the discontinuous feedback generates a well-posed Filippov flow.

The formulation nests both established levers as limits. With the transfer inactive it is information control, and supplies what that lever has lacked: a state, the belief pair; an input, the disclosure gain; a state equation, the Riccati flow; a quadratic input cost; and a value verified in the viscosity sense. What is steered is the agents' uncertainty rather than a physical variable. With the disclosure gain fixed it is a continuous-time marginal-contribution mechanism.

The multi-area power systems instantiation prices both levers in one experiment: the transfer removes $7.4\%$ of social cost and disclosure, optimized jointly, $8.7\%$; where forecast uncertainty rather than capacity binds, disclosure is worth $37\%$ under autarky and $48\%$ under coupling, information being unable to manufacture capacity. Any system with a committed coordinator, transferable utility, observable disruptions, and agents pooling a shared buffer instantiates \eqref{eq:IMC}. 

\appendix
\section{Technical Proofs}

\subsection{Proof of Theorem~\ref{thm:nash_exist} (Equilibrium Existence)}\label{proof:theorem_nash_exist_proof}
\emph{(i) Existence of a social optimizer.} Fix $\rho\in\Ucal_L$ and condition on the observed epoch path (Lemma~\ref{lem:reduction}), so $\Pi$ is a known coefficient on each inter-jump interval. The social problem is a standard finite-horizon stochastic control problem with state $(\hat X,Y)$, compact convex control sets \eqref{eq:caps}, coefficients affine in $u$, and running cost convex in $u$ (Assumption~\ref{ass:data}); its Hamiltonian therefore attains a measurable pointwise minimizer (Berge's maximum theorem plus the Kuratowski--Ryll-Nardzewski measurable-selection theorem). By Proposition~\ref{prop:single_agent}, the value function is the unique viscosity solution of its HJB equation, an optimal Markov feedback $u^\star$ exists, and the closed loop is well posed with an a.e.\ single-valued feedback. Proposition~\ref{prop:single_agent} is established directly from the external theory of viscosity solutions for the single-agent problem, with no equilibrium or incentive construction, so it precedes and is independent of the present theorem.

\emph{(ii) $u^\star$ is a Nash equilibrium.} Under the Groves transfer \eqref{eq:groves}, Agent $a$'s transfer-adjusted objective \eqref{eq:cost} is its own cost minus the pivot $M_{a}$, and $M_{a}$ is built from the \emph{rivals'} value functions $V_b$, $b\neq a$, and their no-coupling baselines, both independent of $a$'s own control. Adding the control-independent constant $R_{a}$ and the rivals' costs (which $a$ does not choose), $a$'s objective equals the social cost $\E\int_0^T\ell\,\diff t$ up to terms independent of $u_a$. Hence the lower level is a \emph{potential game} with potential $\E\int_0^T\ell\,\diff t$, and the Nash property follows from the next lemma.

\begin{lemma}[Partial-minimization consistency]\label{lem:partialmin}
Let the admissible class factor as a product across agents, $\Ucal_O=\prod_a\Ucal_a$ (as it does here: the constraints \eqref{eq:caps} bind componentwise and adaptedness is imposed per agent), and let $\Dcal:\Ucal_O\to\R$ be any functional with global minimizer $u^\star$. Then for each $a$, $u^\star_a$ minimizes $u_a\mapsto \Dcal(u_a,u^\star_{-a})$ over $\Ucal_a$. If in addition $\Dcal$ is induced by a running cost whose Hamiltonian is jointly convex in $u$, strictly convex in the block $u^{\mathrm c}$, and linear in the block $u^{\mathrm l}$, then off the gradient-tie switching set $\{\nabla_{Y_a}S=\nabla_{Y_b}S\}$---which is Lebesgue-null along the closed loop by Proposition~\ref{prop:single_agent}(c)---the pointwise joint Hamiltonian minimizer is unique and its $u_a$-component coincides with the unique minimizer of the Hamiltonian in $u_a$ at fixed $u^\star_{-a}$.
\end{lemma}
\begin{proof}
The first claim requires no convexity: for any $v_a\in\Ucal_a$, the profile $(v_a,u^\star_{-a})$ lies in $\Ucal_O$ by the product structure, so $\Dcal(v_a,u^\star_{-a})\ge \Dcal(u^\star)=\Dcal(u^\star_a,u^\star_{-a})$. For the second claim, off the switching set the joint minimizer is unique (strict convexity on $u^{\mathrm c}$; the linear program on $u^{\mathrm l}$ has a unique solution at an extreme point of $U_a$ determined by the strict sign condition), the first-order conditions $\nabla_{u_a}\Hcal=0$ (or the corresponding vertex/projection conditions at the box boundary) hold blockwise at the joint minimizer, and convexity in $u_a$ at fixed $u^\star_{-a}$ makes these same conditions sufficient for the partial problem; uniqueness of the partial minimizer follows from the same strict convexity on $u_a^{\mathrm c}$ and the strict sign condition on $u_a^{\mathrm l}$.
\end{proof}

\noindent Applying the lemma \ref{lem:partialmin} with $\Dcal(u)=\E\int_0^T\ell\,\diff t$: no unilateral deviation improves $a$'s transfer-adjusted objective, i.e.\ $u^\star$ is a Nash feedback equilibrium. Note the lemma's hypotheses are exactly our structure: joint \emph{strict} convexity fails (the Hamiltonian is linear on the block $u^{\mathrm l}$, which is what produces the bang-bang feedback), and the product form of the admissible class is what carries the first claim.
\emph{(iii) Uniqueness.} Strict convexity of $c_a$ in $u_a$ makes the pointwise Hamiltonian minimizer unique on the strictly convex block. The linear block enters affinely, so at gradient ties $\nabla_{Y_a}S=\nabla_{Y_b}S$ any tie-breaking is optimal; by Proposition~\ref{prop:single_agent}(c) this set is Lebesgue-null along the closed loop, giving uniqueness up to null modification. $\qed$

\subsection{Proof of Theorem~\ref{thm:ic} (Groves Incentive Compatibility)}\label{proof:theorem_ic_proof}
\emph{Setup and well-posedness.} Fix an arbitrary profile of rivals' reports $\hat c_{-a}$. Given the report vector $(\hat c_a,\hat c_{-a})$, the coordinator's implemented feedback $u^{\hat c}$ is the socially optimal Markov feedback for the reported data. By Proposition~\ref{prop:single_agent} applied to the reported-data problem (its hypotheses hold for any admissible report, since messages are $C^1$ strictly convex costs), this feedback exists, is essentially unique, and the associated value functions are the unique viscosity solutions of their HJB equations, hence well-defined non-anticipative functions of the current state $(t,\hat X,Y)$ only. In particular $V_b^{\text{NoCoupling}}$ is the value of $b$'s \emph{autarkic} single-agent problem and depends only on $b$'s report; the pivot $R_{a,t}$ is report-independent by construction. Transfers are $\Gfilt_t$-adapted and integrable. Indeed, the local state is bounded pathwise: from \eqref{eq:Ydyn}, the boundedness of $f_a$ on the compact admissible sets of Assumption~\ref{ass:data} gives $|\dot Y_{a,t}|\le \sup\{|f_a(t,\hat X,Y,u)|: u\in U(\hat X,\Nu)\}=:C_Y$, so $|Y_{a,t}|\le|Y_{a,0}|+C_Y T$ deterministically. For the belief mean, \eqref{eq:meanfilter} gives $|\hat X_t|\le|\hat X_0|+\int_0^t\big(\|A\|_\infty|\hat X_s|+\|B\|_\infty|\hat Q_s|\big)\diff s+\big|\!\int_0^t\Pi_s\rho_s^\top R^{-1/2}\diff\beta_s\big|$, where the stochastic-integral integrand is bounded by $\bar\pi\bar\rho\|R^{-1/2}\|$ (Lemma~\ref{lem:riccati}); taking $\sup_{t\le T}$, squaring, applying the Burkholder--Davis--Gundy inequality to the martingale term and Gr\"onwall's lemma to the drift yields $\E\big[\sup_{t\le T}|\hat X_t|^2\big]\le C\big(1+\E|\hat X_0|^2+\E\sup_{t\le T}|\hat Q_t|^2\big)\exp(CT)<\infty$, with the same bound for the Kalman--Bucy estimate $\hat F$ of the Ornstein--Uhlenbeck factor. Since the value functions are continuous with quadratic growth (Proposition~\ref{prop:single_agent}(a)), $\E\sup_{t\le T}|V_b(t,\hat X_t,Y_t)|<\infty$ and the transfers \eqref{eq:groves} are integrable on $[0,T]$, pathwise bounded on each inter-jump interval.

\emph{(i) Reporting dominance.} Write $\Dcal^{\text{true}}_a(u)$ for $a$'s expected true cost under an action profile $u$, and $\Dcal^{\hat c}_{\text{soc}}(u)=\Dcal^{\hat c_a}_a(u)+\sum_{b\neq a}\Dcal^{\hat c_b}_b(u)$ for the reported social cost. Under \eqref{eq:groves}, $a$'s expected net cost when reporting $\hat c_a$ is

\begin{equation}\label{eq:ic_decomp}
\Dcal^{\text{true}}_a\big(u^{\hat c}\big) + \sum_{b\neq a} \Dcal^{\hat c_b}_b\big(u^{\hat c}\big)\;-\;\sum_{b\neq a} V_b^{\text{NoCoupling}}(\hat c_b) +\;\E\!\int_0^T R_{a,t}\,\diff t,
\end{equation}

where the last two terms are independent of $\hat c_a$ (the baselines depend only on rivals' reports; the pivot by construction). The first two terms equal the \emph{mixed} social cost $\Dcal^{(c_a,\hat c_{-a})}_{\text{soc}}(u^{\hat c})$: the total cost of the implemented profile measured with $a$'s \emph{true} cost and the rivals' reported costs. The implemented profile $u^{\hat c}$ minimizes the reported social cost $\Dcal^{(\hat c_a,\hat c_{-a})}_{\text{soc}}$ \emph{over the class of $\Gfilt_t$-adapted Markov feedbacks}: this is Theorem~\ref{thm:nash_exist}(i) applied to the reported data, whose hypotheses hold because the message space consists of $C^1$, strictly convex cost functions, so Assumption~\ref{ass:data} is satisfied by any admissible report profile. Moreover the \emph{mixed} criterion $\Dcal^{(c_a,\hat c_{-a})}_{\text{soc}}$ is itself convex in $u$---it is the sum of $a$'s true convex cost and the rivals' reported convex costs, over the same convex feasible set---so the mixed problem is well-posed by the same theorem. Setting $\hat c_a=c_a$ makes the implemented profile the exact minimizer of the mixed criterion that $a$'s net cost \eqref{eq:ic_decomp} tracks. For any misreport, $u^{(c^D_a,\hat c_{-a})}$ is feasible for the mixed criterion but generally suboptimal for it, so
$\Dcal^{(c_a,\hat c_{-a})}_{\text{soc}}\big(u^{(c^D_a,\hat c_{-a})}\big)\ \ge\ \Dcal^{(c_a,\hat c_{-a})}_{\text{soc}}\big(u^{(c_a,\hat c_{-a})}\big)$,
with the inequality holding \emph{pointwise in the rivals' report profile}---no step above uses truthfulness of the rivals---and hence in expectation under any belief $a$ may hold about them. Truthful reporting is therefore a dominant strategy, and strict suboptimality obtains whenever the misreport changes the implemented profile on a set of positive measure.

\emph{(ii) Obedience.} Given truthful reports, Theorem~\ref{thm:nash_exist}(ii) shows the efficient profile $u^\star$ is a Nash feedback equilibrium of the transfer-adjusted game: each agent's objective differs from the social potential by control-independent terms, and partial minimization at the joint minimum of the convex potential yields unilateral optimality. A unilateral Markov deviation by $a$ changes $a$'s net cost by the corresponding change in the social potential, which is nonnegative by optimality of $u^\star$. The restriction to Markov deviations is without loss of generality: for the convex-Hamiltonian problem at hand, the value over general $\Gfilt_t$-progressively measurable controls coincides with the value over Markov feedbacks, since by Proposition~\ref{prop:single_agent} the viscosity value bounds every progressive control below and is attained by the Markov feedback---so no progressively measurable deviation can improve on $u^\star_a$ either.

\emph{(iii) Uniqueness up to null sets.} By Proposition~\ref{prop:single_agent}(b)--(c) applied to the reported-data planner problem, the socially optimal feedback is unique wherever the value is $Y$-differentiable, which is Lebesgue-a.e.; the strictly convex components solve a strictly convex program and are pointwise unique, while the transfer components are bang-bang and are pinned down except on the switching set, which is Lebesgue-null. Two implemented feedbacks thus agree a.e.\ in $(t,\hat X,Y)$. Expected running cost and transfers are time integrals of the feedback along the closed loop, hence unchanged by modification on a Lebesgue-null set; the two feedbacks therefore yield identical expected running cost and transfers, so the implemented equilibrium is payoff-unique. $\qed$

\subsection{Proof of Theorem~\ref{thm:master_robust} (Master HJB)}\label{proof:theorem_master_proof}
\emph{Step 1: the controlled state is Markov under $\Gfilt_t$.}
Fix $\rho\in\Ucal_L$. By Lemma~\ref{lem:proj_filter} the conditional law of $X_t$ given $\Gfilt_t$ is exactly $\mathcal N(\hat X_t,\Pi_t)$, so $(\hat X_t,\Pi_t)$ is a sufficient statistic and no information is lost by replacing the partially observed problem \eqref{eq:problemP} with the fully observed problem in the variables $(\hat X,\Pi,Y)$. By Lemma~\ref{lem:reduction}, conditionally on the observed epoch path $\Nu_t$ the process $\Pi$ is the unique solution of an ODE and is independent of $\beta$, of the state, and of the agents' controls. Together with \eqref{eq:Ydyn} this makes $w_t:=(\hat X_t,\Pi_t,Y_t)$ a controlled Markov process with respect to $\Gfilt_t$, with generator computed in Step 3.

\emph{Step 2: dynamic-programming principle.}
The running cost is continuous and, by Assumption~\ref{ass:data}, the agent control set of \eqref{eq:caps} is compact and state-dependent in a Lipschitz, hence measurable, way; the disclosure set $\Ucal_L$ is compact by Assumption~\ref{ass:riccati}. The value \eqref{eq:value} is therefore finite and, by Berge's maximum theorem together with the Kuratowski--Ryll-Nardzewski selection theorem, the pointwise infimum over $\rho$ is attained by a measurable selector. The DPP
\[
S(t,w)=\inf_{\rho\in\Ucal_L}\ \E\Big[\int_t^{t+h}\!\!\ell\,\diff s+S\big(t+h,w_{t+h}\big)\;\Big|\;w_t=w\Big]
\]
then follows by the standard argument for controlled Markov processes~\cite{FlemingSoner,YongZhou}, valid for $h>0$ small.

\emph{Step 3: the generator, including the epoch term.}
The state $w$ is a piecewise-deterministic jump-diffusion: between epochs $\hat X$ solves \eqref{eq:meanfilter} with $\Gfilt_t$-innovation diffusion coefficient $G_\rho=\Pi\rho^\top R^{-1/2}$ (Lemma~\ref{lem:proj_filter}), $\Pi$ follows the Riccati drift $\mathcal R(\Pi,\rho)$, and $Y$ follows \eqref{eq:Ydyn}; at each epoch, which arrives with intensity $\lambda$, the mean is unchanged and the covariance inflates by $\Sigma_J$ (Section~\ref{sec:signal}). Applying the It\^o--L\'evy formula for such a process to $S\in C^{1,2}(\mathcal D)$ over $[t,t+h]$,

\begin{equation*}
S(t{+}h,w_{t+h}) - S(t,w)
= \int_t^{t+h}\!\!\Big(\partial_s S + \Lcal^\rho S\Big)\diff s + M_{t,t+h},
\end{equation*}

where the generator is
\begin{align*}
\Lcal^\rho S &= \nabla_{\hat X}S^\top\!\big(A_s\hat X+B\hat Q_s\big)
+\tfrac12\Tr\!\big(\nabla^2_{\hat X}S\,G_\rho G_\rho^\top\big) \\
&\quad+\Tr\!\big(\partial_\Pi S\,\mathcal R(\Pi,\rho)\big)
+\textstyle\sum_a \nabla_{Y_a}S^\top \dot Y_a
+\lambda\,\Ical S,
\end{align*}
with $G_\rho G_\rho^\top=\Pi\rho^\top R^{-1}\rho\Pi$, and where the nonlocal term arises as the compensator of the epoch process: the epochs form a Poisson process of intensity $\lambda$ and each maps $\Pi\mapsto\Pi+\Sigma_J$ deterministically, so the compensated jump measure contributes exactly $\lambda\big[S(\cdot,\Pi+\Sigma_J,\cdot)-S(\cdot,\Pi,\cdot)\big]=\lambda\Ical S$. Since the marks are unobserved and centred, no integral against $\mathcal N(\mathbf 0,\Sigma_J)$ appears in the $\hat X$-argument; as recorded in Section~\ref{sec:master}. Note that $\mathcal R$ carries no $\lambda\Sigma_J$ term: the predictive rate of \eqref{eq:riccati} is the \emph{average} of these resets and would double-count them here.

\emph{Step 4: the local martingale is a true martingale.}
The remainder is
$M_{t,t+h}=\int_t^{t+h}\nabla_{\hat X}S^\top G_\rho\,\diff\beta_s + \widetilde M_{t,t+h}$,
with $\widetilde M$ the compensated epoch martingale. By Lemma~\ref{lem:riccati}, $\underline\pi I\preceq\Pi_s\preceq\bar\pi I$ and $\|\rho_s\|\le\bar\rho$, so $\|G_\rho\|\le\bar\pi\bar\rho\|R^{-1/2}\|$ is bounded. Since $S$ has polynomial growth and $\E[\sup_{s\le T}|w_s|^2]<\infty$ by the estimate in the proof of Theorem~\ref{thm:ic} (bounded $Y$ pathwise; Burkholder--Davis--Gundy and Gr\"onwall for $\hat X$), the integrand is square-integrable and the stochastic integral has zero expectation. The epoch martingale is a compensated jump martingale with bounded jumps, $|\Ical S|$ being bounded on compacts by the polynomial growth of $S$ and the boundedness of $\Sigma_J$, and with at most $K$ jumps per horizon (Assumption~\ref{ass:riccati}); it is therefore a true martingale of zero mean. Hence $\E[M_{t,t+h}]=0$.

\emph{Step 5: passage to the limit.}
Taking expectations in Step 3, inserting into the DPP of Step 2, and dividing by $h$,
\[
0=\inf_{\rho\in\Ucal_L}\ \frac1h\,\E\int_t^{t+h}\Big(\ell(s,w_s,u^\star_s,\rho)+\partial_s S+\Lcal^{\rho}S\Big)\diff s .
\]
The integrand is continuous in $s$ and uniformly bounded on compacts, uniformly in $\rho$ over the compact $\Ucal_L$; letting $h\downarrow0$ and using continuity of the infimum over a compact set (Berge) gives
$\partial_t S+\inf_{\rho\in\Ucal_L}\{\ell+\Lcal^\rho S\}=0$ at $(t,w)$, which is \eqref{eq:isaacs} once $\ell$ is written as $\sum_a\mathcal C^{\mathrm{macro}}_a$ and the $\rho$-independent term $\lambda\Ical S$ is taken outside the minimization. The terminal condition is immediate from $\eqref{eq:value}$ at $t=T$. The infimum is attained because $\Ucal_L$ is compact and the bracket is continuous in $\rho$, so it may be written as a minimum.$\qed$

\bibliographystyle{cas-model2-names}

\bibliography{references}

\end{document}